\documentclass[11pt]{article}
\usepackage[T1]{fontenc}
\usepackage{times} 
\usepackage{textcomp}
\usepackage{amsmath,amsthm,amssymb} 
\usepackage{bbm}           

\usepackage{mathrsfs}

\newtheorem{thrm}{Theorem}[section]
\newtheorem{lmm}[thrm]{Lemma}
\newtheorem{prpstn}[thrm]{Proposition}
\newtheorem{crllr}[thrm]{Corollary}

\newtheorem{assumption}[thrm]{Assumption}
\newtheorem{dfntn}[thrm]{Definition}

\DeclareMathOperator*{\argmin}{arg\,min}

\newcommand{\Rb}{\mathbbm{R}}

\newcommand{\1}{\mathbbm{1}}

\newcommand{\Dc}{\mathcal{D}}
\newcommand{\Fb}{\mathbb{F}}

\newcommand{\Eb}{\mathbb{E}}
\newcommand{\Lb}{\mathbb{L}}

\newcommand{\Sc}{\text{S}}

\newcommand{\Hc}{\text{H}}
\newcommand{\Cc}{\mathcal{C}}

\newcommand{\xLtwo}{\text{L}^2}
\newcommand{\xdif}{\text{\,d}}
\newcommand{\Rref}{\ref}

\newenvironment{tightitemize}{%
    \list{{\textup{$\bullet$}}}{\settowidth\labelwidth{{\textup{\qquad}}}
    \leftmargin\labelwidth \advance\leftmargin\labelsep
    \parsep 0pt plus 1pt minus 1pt \topsep 3pt \itemsep 3pt
    }}{\endlist}

\newenvironment{tightlist}[1]{%
    \list{{\textup{(\roman{enumi})}}}{\settowidth\labelwidth{{\textup{(#1)}}}
    \leftmargin 6pt \advance\leftmargin\labelsep \itemindent \parindent
    \parsep 0pt plus 1pt minus 1pt \topsep 0pt \itemsep 0pt
    \usecounter{enumi}}}{\endlist}

\title{Discrete-Time Approximation of Risk-Averse Control Problems for Diffusion Processes}
\author{
 Andrzej Ruszczy\'nski\thanks{
Rutgers University, Department of Management Science and Information Systems, Piscataway, NJ 08854, USA, Email: {rusz@rutgers.edu}}
\and
Jianing Yao\thanks{
Rutgers University, Department of Management Science and Information Systems, Piscataway, NJ 08854, USA, Email: {jy346@rutgers.edu}}
}

\begin{document}
\maketitle


\begin{abstract}
We consider optimal control problems for diffusion processes, where the objective functional is defined by a time-consistent dynamic risk measure. We focus
on coherent risk measures defined by $g$-evaluations. For such problems, we construct a family of time and space perturbed systems with piecewise-constant control
functions. We obtain a regularized optimal value function by
a special mollification procedure. This allows us to establish a bound on the difference between the optimal value functions of the original problem and of the problem with piecewise-constant controls.
\end{abstract}



\section{Introduction}
\label{s:intro}
The first introduction of a coherent (static) risk measure, by Artzner \emph{et al.} \cite{AP1,AP2}, was motivated by the capital adequacy rules of the Basel Accord. Large volume of research were devoted to this area,  F\"{o}llmer, Schied \cite{SF1}  and Frittelli and Rosazza Gianin\cite{FR1} generalized it to convex risk measure, Ruszczy\'{n}ski and Shapiro \cite{AR1} studied it from the perspective of optimization. Several classical references concerning static risk measures are \cite{SF2,FS,AR2,AR3}.

Further development of the theory of  risk measures lead  to a dynamic setting, in which the risk  is measured at each time instance based on the updated information. The key condition of \emph{time-consistency} allows for dynamic programming formulations. The discrete time case was extensively explored by Detlfsen and Scandolo \cite{DS1}, Bion-Nadal \cite{BN}, Cheridito et al. \cite{CD1,CK1}, F\"{o}llmer and Penner\cite{FP1}, Frittelli and Scandolo \cite{FR2},  Riedel \cite{RF}, and Ruszczy\'{n}ski and Shapiro\cite{AR2}. For the continuous-time case, Coquet,  Hu,  M{\'e}min and Peng \cite{coquet2002filtration} discovered that time-consistent dynamic risk measures can be represented as solutions of \emph{Backward Stochastic Differential Equations} (BSDE) (see also \cite{PSG1,gianin2006risk}). Inspired by that, Barrieu and El Karoui provided a comprehensive study in \cite{BE1,BE2};  further contributions being made by Delbean, Peng, and Rosazza Gianin \cite{delbaen2010representation}, and Quenez and Sulem \cite{quenez2013bsdes} (for a more general model with Levy processes). In addition, application to finance was considered, for example, in \cite{laeven2014robust}. Using the convergence results of Briand, Delyon and M\'emin \cite{BDM}, Stadje \cite{ST} finds the drivers of BSDE corresponding to discrete-time risk measures.

As for control with risk aversion, in discrete time setting, Ruszczy\'{n}ski \cite{AR4}, \c{C}avu\c{s} and Ruszczy\'nski \cite{cavus2014risk} and Fan and Ruszczy\'nski \cite{fan2015dynamic} developed the concept of a Markov risk measure and proposed risk-averse dynamic programming equations as well as computation methods. Our intention is to use continuous-time dynamic risk measures as objective functionals in optimal control problems for diffusion processes. While the traditional continuous stochastic control is well developed and discussed in numerous books (see, e.g., \cite{FleSon,Krylov,Pham,XYZ}), the risk-averse case
appears to be largely unexplored. In the present paper, we consider the risk-averse case with coherent risk measures given by $g$-evaluations. Such control problems are closely related to forward--backward systems of stochastic differential equations (FBSDE) (see, \cite{MA,peng1999fully}). For controlled fully coupled FBSDEs, Li and Wei \cite{JLQW} obtained the dynamic programming equation and derived the corresponding Hamilton--Jacobi--Bellman equation. Maximum principle for forward--backward systems and corresponding games was derived in \cite{oksendal2009maximum,oksendal2014forward}, including models with Levy processes.

The  contribution of this paper is the study of accuracy of discrete-time approximations of risk-averse optimal control problems with coherent risk measures given by $g$-evaluations. For the
purpose of the study, we construct a family of perturbed systems with two types
of perturbations: of the initial time and the initial state. For such a family,
we integrate the value functions of a piecewise-constant control with respect
to the said initial time and state values. This yields regularized functions
for which It\^{o} calculus can be applied. Using the earlier results on the
Hamilton--Jacobi--Bellman equation for risk-averse problems, we establish
an error bound of order $\Delta^{1/6}$, between the optimal values
of the original system and a system with piecewise constant controls with
time step $\Delta$.

Section \S \ref{s:foundations} has a synthetic character. We review in it the concept of $\Fb$-consistent evaluations and the connections to backward stochastic differential equations and dynamic risk measures. In \S \ref{s:cumul}, we formulate the
 risk-averse optimal control problem and study its basic properties. In the meanwhile, we recall the dynamic programming equation and the risk-averse analog of the Hamilton-Jacobi-Bellman equation. In section \S \ref{s:perturbed} we construct a family of time and space perturbed problems. They are used in a specially designed mollification procedure in \S \ref{s:mollification}, which yields sufficiently smooth close approximations of the optimal value function. In \S \ref{s:accuracy}, we prove that  the accuracy of the control policies restricted to piecewise-constant controls is of the order $h^{1/3}$, where $h^2$ is the time discretization step.

\section{Foundations}
\label{s:foundations}
\subsection{Nonlinear Expectations and Dynamic Risk Measures}
\label{s:ne}
We establish a suitable framework and briefly review the concept of $\Fb$-consistent nonlinear expectations (for an extensive treatment, see  \cite{PSG1}). For $0<T<\infty$, let $(\varOmega,\mathcal{F},\mathbb{P},\Fb)$ be a probability space, where $\Fb=\left\{\mathcal{F}_t\right\}_{0 \le t \le T}$ is a filtration. A vector-valued stochastic process $\{X_t\}_{0 \le t \le T}$ is said to be adapted to $\Fb$ if $X_t$ is an $\mathcal{F}_t$-measurable random variable for any $t\in [0,T]$.

We introduce the following notation.
\begin{tightitemize}
\item $\mathbb{E}_{t}[\,\cdot\,] : = \mathbb{E}[\,\cdot\,|\,\mathcal{F}_t]$;
\item $\text{P}^m[t,T]$: the set of $\Rb^m$-valued adapted processes on $[t, T]\times \varOmega$;
\item $\xLtwo(\varOmega, \mathcal{F}_t, \mathbb{P};\Rb^m)$: the set of $\Rb^m$-valued  $\mathcal{F}_t$-measurable random variables $\xi$ such that $\| \xi \|^2 := \mathbb{E}[\,|\xi|^2\,]<\infty$;
    for $m=1$, we write it $\xLtwo(\varOmega, \mathcal{F},\mathbb{P})$;
\item $\Sc^{2,m}[t, T]$: the set of elements $Y\in \text{P}^m[t,T]$ such that
$
\| Y \|^2_{\Sc^{2,m}[t, T]} :=  \mathbb{E} [\,\sup_{\,t\leq s\leq T}|Y_s|^2\,] < \infty$;
 for $m=1$, we write it $\Sc^{2}[t, T]$;
\item $\Hc^{2,m}[t, T]$: the set of elements $Y\in \text{P}^m[t,T]$, such that
$
\| Y \|^2_{\Hc^{2,m}[t, T]} := \mathbb{E}\Big[ \int_t^T |Y_s|^2\xdif s\Big] < \infty$;
 for $m=1$ we write it $\Hc^{2}[t, T]$;\footnote{When the norm is clear from the context, the subscripts are skipped.}
\item $\Cc^{1,2}([t,T]\times\Rb^m)$ the space of functions $f:[t,T]\times\Rb^m\to\Rb$, which are differentiable with respect to the first argument and twice differentiable with respect to the second argument, with all these derivatives continuous with respect to both arguments;
\item $\Cc^{1,2}_{\textup{b}}([t,T]\times\Rb^m)$ the space of functions $f\in \Cc^{1,2}([t,T]\times\Rb^m)$ with all derivatives bounded and continuous with respect to both arguments;
\item $\Cc^\infty(B)$: the space of functions $f:B\to\mathbb{R}$ that are infinitely continuously differentiable with respect to all arguments and have compact support on $B\subset \mathbb{R}^n$.
\end{tightitemize}

With this notation, we can introduce the concept of a nonlinear expectation.
\begin{dfntn} For\, $0 \le T< \infty$, a \emph{nonlinear expectation} is a functional
$
\rho_{0,T}: \xLtwo(\varOmega,\mathcal{F}_T, \mathbb{P}) \to \Rb
$
 satisfying the strict monotonicity property:
\begin{equation*}
\begin{split}
&\mbox{if } \xi_1\geq \xi_2 \mbox{ } \mbox{a.s.}, \mbox{ then } \rho_{0,T}[\,\xi_1\,] \geq \rho_{0,T}[\,\xi_2\,];\\
&\mbox{if } \xi_1\geq \xi_2 \mbox{ } \mbox{a.s.},  \text{ then } \rho_{0,T}[\,\xi_1\,] = \rho_{0,T}[\,\xi_2\,]\text{ if and only if } \xi_1 = \xi_2 \mbox{ a.s.};
\end{split}
\end{equation*}
and the constant preservation property:
\[
\rho_{0,T}[\,c\1_{\varOmega}\,] = c,\quad \forall\;c\in \Rb,
\]
where $\1_A$ is the characteristic function of the event $A\in\mathcal{F}_T$.
\end{dfntn}
Based on that, the $\Fb$-consistent nonlinear expectation is defined as follows.
\begin{dfntn}
\label{d:consistent}
For a filtered probability space $(\varOmega,\mathcal{F}, \mathbb{P}, \mathbb{F})$, a nonlinear expectation $\rho_{0,T}[\,\cdot\,]$ is  \emph{$\Fb$-consistent} if for
every $\xi\in \xLtwo(\varOmega, \mathcal{F}_T, \mathbb{P})$ and every $t\in [0,T]$  a random variable $\eta\in \xLtwo(\varOmega, \mathcal{F}_t, \mathbb{P})$ exists such that
\begin{equation*}
\rho_{0,T}[\,\xi\1_A\,]=\rho_{0,T}[\,\eta \1_A\,]\quad \forall A\in \mathcal{F}_t.
\end{equation*}
\end{dfntn}

The variable $\eta$ in Definition \ref{d:consistent} is uniquely defined, we denote it by $\rho_{t,T}[\,\xi\,]$. It can  be interpreted as a nonlinear conditional expectation of $\xi$ at time t. We can now define for every $t\in [0,T]$ the
corresponding nonlinear expectation $\rho_{0,t}: \xLtwo(\varOmega,\mathcal{F}_t, \mathbb{P}) \to \Rb$ as follows:
$ \rho_{0,t}[\,\xi\,] = \rho_{0,T}[\,\xi\,]$, for all $\xi\in \xLtwo(\varOmega, \mathcal{F}_t, \mathbb{P})$. In this way, a whole system
of $\Fb$-consistent nonlinear expectations $\big\{ \rho_{s,t}\big\}_{0\le s \le t \le T}$ is defined.
\begin{prpstn}
\label{p:nonev}
If $\rho_{0,T}[\,\cdot\,]$ is an $\Fb$-consistent nonlinear expectation, then for all $0\leq t\leq T$ and all $\xi, \xi' \in \xLtwo(\varOmega,\mathcal{F}_T, \mathbb{P})$, it has the following properties:
\begin{tightlist}{iii}
\item \textbf{Generalized constant preservation}: If $\xi\in \xLtwo(\varOmega,\mathcal{F}_t, \mathbb{P})$, then $\rho_{t,t}[\,\xi\,]=\xi$;
\item \textbf{Time consistency}: $\rho_{s,T}[\,\xi\,] = \rho_{s,t}[\,\rho_{t,T}[\,\xi\,]\,]$, for all  $0\le s\leq t$;
\item \textbf{Local property}:
 $\rho_{t,T}[\,\xi\1_A+\xi'\1_{A^c}\,]= \1_A \rho_{t,T}[\,\xi\,]+\1_{A^c} \rho_{t,T}[\,\xi'\,]$, for  all $A\in \mathcal{F}_t$.
\end{tightlist}
\end{prpstn}
It follows that $\mathbb{F}$-consistent nonlinear expectations are special cases of dynamic time-consistent measures of risk, enjoying
a number of useful properties. They do not, however, have the properties of convexity, translation invariance, or positive homogeneity,
unless additional assumptions are made. We shall return to this issue in the next subsection.

\subsection{Backward Stochastic Differential Equations and $g$-Evaluations}
\label{s:BSDE}
Close relation exists between $\Fb$-consistent nonlinear expectations on the space $\xLtwo(\varOmega,\mathcal{F}, \mathbb{P})$,
with the natural filtration of the Brownian motion,
 and backward stochastic differential equations (BSDE) \cite{PP1,PP2,PSG2}.  We equip $(\varOmega,\mathcal{F}, \mathbb{P})$ with a $d$-dimensional Brownian filtration, i.e., $\mathcal{F}_t = \sigma \big\{\{W_s; 0\leq s\leq T\}\cup \mathcal{N}\big\}$,
\noindent where $\mathcal{N}$ is the collection of $P$-null sets in $\varOmega$. In this paper we consider the following one-dimensional BSDE:
\begin{equation}
\label{eq:BSDE}
-\xdif Y_t=g(t, Y_t, Z_t)\xdif t - Z_t \xdif W_t, \quad Y_T = \xi,
\end{equation}
where  the data is the pair $(\xi, g)$, called the  \emph{terminal condition} and the \emph{generator} (or \emph{driver}), respectively.
Here, $\xi\in \xLtwo(\varOmega, \mathcal{F}_T, \mathbb{P})$, and $g:[0,T]\times \Rb\times \Rb^d \times\varOmega\rightarrow \Rb$ is a
measurable function (with respect to the product $\sigma$-algebra), which is \emph{nonanticipative}, that is,
$g(t, Y_t, Z_t)$ is $\mathcal{F}_t$-measurable for all $t\in [0,T]$.

The solution of the BSDE is a pair of processes $(Y, Z)\in \Sc^2[0, T]\times \Hc^{2,d}[0, T]$ such that
\begin{equation}
\label{bsint}
Y_t=\xi+\int^T_t\! g(s,Y_s,Z_s)\xdif s -\int^T_t\! Z_s\xdif W_s,\quad  t\in[0, T].
\end{equation}
The existence and uniqueness of the solution of (\ref{eq:BSDE}) can be guaranteed under the following assumption.
\begin{assumption}[Peng and Pardoux \cite{PP1}]
\label{a:pp1}
\begin{tightlist}{ii}
\item $g$ is jointly Lipschitz in $(y, z)$, i.e., a constant $K > 0$ exists such that for all $t\in [0,T]$, all $y_1,y_2\in \Rb$ and all $z_1,z_2\in \Rb^{d}$ we have
\begin{equation*}
|g(t, y_1, z_1)-g(t, y_2, z_2)| \leq K(|y_1-y_2|+|z_1-z_2|) \quad \textit{a.s.; }
\end{equation*}
\item the process $g(\cdot, 0, 0) \in \Hc^2[0, T]$.
\end{tightlist}
\end{assumption}

Under Assumption \ref{a:pp1}, we can define the $g$-evaluation.
\begin{dfntn}
\label{d:gev}
For each $0\leq t\leq T$ and $\xi\in \xLtwo(\varOmega, \mathcal{F}_T, \mathbb{P})$, the \emph{g-evaluation}  at time $t$ is the operator $\rho^g_{t,T}: \xLtwo(\varOmega, \mathcal{F}_T, \mathbb{P})\to  \xLtwo(\varOmega, \mathcal{F}_t, \mathbb{P}) $
defined as follows:
\begin{equation}
\label{g-evaluation}
\rho^g_{t,T}[\,\xi\,] = Y_t,
\end{equation}
where $(Y, Z)\in \Sc^{2,d}[t, T]\times \Hc^2[t, T]$ is the unique solution of \eqref{eq:BSDE}.
\end{dfntn}
The following theorem reveals the relationship between $g$-evaluation and $\Fb$-consistent nonlinear expectation.
\begin{thrm}
\label{t:g-evaluation}
Let the driver $g$ satisfy Assumption \Rref{a:pp1} and the condition:  $g(\cdot, 0, 0) \equiv 0$ a.s.. Then the system of $g$-evaluations $\big(\rho^g_{t,T}\big)_{0\le t \le T}$ defined in \eqref{g-evaluation} is a system of  $\Fb$-consistent nonlinear expectations.
Furthermore, we have
\begin{equation*}
\lim_{s\uparrow t}\rho^g_{s,t}[\,\xi\,] = \xi,\quad \forall\, \xi\in \xLtwo(\varOmega, \mathcal{F}_t,\mathbb{P}), \, t\in [0,T].
\end{equation*}
\end{thrm}

Surprisingly, Coquet,  Hu,  M{\'e}min, and Peng proved in \cite{coquet2002filtration} that every $\Fb$-consistent nonlinear expectation
which is ``dominated''  by $\rho^{\mu, \nu}_{0,T}$ (a $g$-evaluation with the driver $\mu |y| + \nu| z|$ with some $\nu, \mu >0$) is in fact
a  $g$-evaluation with some~$g$. The domination is understood as follows:
$\rho_{0,T}[Y+\eta] - \rho_{0,T}[Y] \le \rho^{\mu,\nu}_{0,T}[\eta]$, for all $Y$, $\eta\in \xLtwo(\varOmega, \mathcal{F}_T,\mathbb{P})$.\\

From now on we shall use only $g$-evaluations as time-consistent dynamic measures of risk. To ensure desirable properties
of the resulting measures of risk, we shall impose additional conditions on the driver $g$.

\begin{assumption}
\label{a:riskmeasure}
The driver $g$ satisfies for almost all $t\in [0,T]$ the following conditions:
\begin{tightlist}{iii}
\item $g$ is deterministic and independent of $y$, that is, $g:[0,T]\times \Rb^d\to\Rb$, and $g(\cdot, 0) \equiv 0$;
\item $g(t, \cdot)$ is convex for all $t\in [0,T]$;
\item $g(t,\cdot)$ is positively homogeneous for all $t\in [0,T]$.
\end{tightlist}
\end{assumption}

Under these conditions, one can derive new properties of the evaluations $\rho_{t,T}^g$, $t\in [0,T]$, in addition to the
general properties of $\Fb$-consistent nonlinear expectations stated in Proposition \ref{p:nonev}.
\begin{thrm}
\label{t:geval-prop}
Suppose $g$ satisfies Assumption \Rref{a:pp1} and condition (i) of Assumption \Rref{a:riskmeasure}.
Then the system of $g$-evaluations $\rho_{t,r}^g$, $0\le t \le r \le T$ has the following properties:
\begin{tightlist}{iii}
\item \textbf{Normalization}: $\rho_{t,r}^g(0)=0$;
\item \textbf{Translation Property}: for all $\xi\in \xLtwo(\varOmega, \mathcal{F}_r, \mathbb{P})$ and $\eta\in \xLtwo(\varOmega, \mathcal{F}_t, \mathbb{P})$,
\[
    \rho_{t,r}^g(\xi+\eta)=\rho_{t,r}^g(\xi)+\eta, \quad \text{a.s.};
\]
    \end{tightlist}
If, additionally, condition (ii) of of Assumption \Rref{a:riskmeasure} is satisfied, then
$\rho_{t,r}^g$ has the following property:
\begin{tightlist}{iii}
\setcounter{enumi}{2}
\item \textbf{Convexity}: for all $\xi,\xi'\in \xLtwo(\varOmega, \mathcal{F}_r, \mathbb{P})$
and  all $\lambda \in L^\infty(\varOmega, \mathcal{F}_t, \mathbb{P})$ such that \mbox{$0 \leq \lambda \leq 1$},
\begin{equation*}
\rho_{t,r}^g(\lambda \xi + (1-\lambda)\xi')\leq \lambda \rho_{t,r}^g(\xi)+(1-\lambda)\rho_{t,r}^g(\xi'),\quad \text{a.s.}.
\end{equation*}
\end{tightlist}
Moreover, if $g$ also satisfies condition (iii) of Assumption \Rref{a:riskmeasure}, then $\rho_{t,r}^g$ has also the following property:
\begin{tightlist}{iv}
\setcounter{enumi}{3}
\item \textbf{Positive Homogeneity}: for all $\xi\in \xLtwo(\varOmega, \mathcal{F}_r, \mathbb{P})$
and all $\beta \in L^\infty(\varOmega, \mathcal{F}_t, \mathbb{P})$ such that $\beta \ge 0$,
we have
\[
\rho_{t,r}^g(\beta \xi) = \beta \rho_{t,r}^g(\xi), \quad \text{a.s.}.
\]
\end{tightlist}
\end{thrm}
It follows that under Assumptions \ref{a:pp1} and  \ref{a:riskmeasure}, the $g$-evaluations $\rho_{t,r}^g$ are convex or coherent conditional measures of risk (depending on whether (iii) is assumed or not).

Finally, we can derive their dual representation, by specializing the general results of \cite{BE2}.
\begin{thrm}
\label{t:geval-dual}
Suppose $g$ satisfies Assumption \Rref{a:pp1} and  \Rref{a:riskmeasure}.
Then for all $0\le t \le r \le T$ and all $\xi\in \xLtwo(\varOmega, \mathcal{F}_r, \mathbb{P})$
we have
\begin{equation}
\label{geval-dual}
\rho_{t,r}^g(\xi) = \sup_{\varGamma \in \mathcal{A}_{t,r}} \Eb\big[\,\varGamma \xi ~\big|~\mathcal{F}_t\,\big ]
\end{equation}
where $\mathcal{A}_{t,r}=\partial \rho_{t,r}^g(0)$ is defined as follows:
\begin{equation}
\label{A-set}
\mathcal{A}_{t,r} = \left\{\,\exp\left(\int_t^r \gamma_s \xdif W_s - \frac{1}{2}\int_t^r\lvert\gamma_s\rvert^2\xdif s\right): \gamma \in \Hc^{2}[t, r], \ \gamma_s \in \partial g(s,0), \ s\in [t,r]\,\right\}.
\end{equation}
\end{thrm}
\begin{crllr}
\label{c:Gamma-bound}
A constant $C$ exists, such that for all $0\le t \le r \le T$ and all $\varGamma_{t,r} \in \mathcal{A}_{t,r}$ we have
\[
\|\varGamma_{t,r} -1\| \le \frac{r-t}{T} e^{CT}.
\]
\end{crllr}
\begin{proof}
It follows from the definition of $\mathcal{A}_{t,r}$ that $\varGamma_{t,r}$ is the solution of the SDE
\[
\xdif\varGamma_{t,s} = \gamma_s \varGamma_{t,s}\xdif W_s, \quad \gamma_s\in \partial g(s,0),\quad s\in [t,r],\quad \varGamma_{t,t}=1.
\]
Using It\^{o} isometry, we obtain the chain of relations
\[
\|\varGamma_{t,r} - 1\| = \int_t^r \|\gamma_s\varGamma_{t,s}\|^2\xdif s\le \int_t^r \|\gamma_s\|^2\|\varGamma_{t,s}\|^2\xdif s
\le \int_t^r \|\gamma_s\|^2\big( 1+ \|\varGamma_{t,s}-1\|^2\big)\xdif s.
\]
If $u$ is a uniform upper bound on the norm of the subgradients of $g(s,0)$ we deduce that
$\|\varGamma_{t,r} - 1\|^2 \le \varDelta_s$, $s\in [t,r]$, where $\varDelta$ satisfies the ODE:
$\frac{\xdif}{\xdif s} \varDelta_s= u(1+\varDelta_s)$, with $\varDelta_t=0$. Consequently,
\[
\|\varGamma_{t,r} - 1\|^2 \le \varDelta_r =  e^{u^2(r-t)} -1.
\]
The convexity of the exponential function yields the postulated bound.
\end{proof}

\section{The Risk-Averse Control Problem}\label{s:cumul}
\subsection{Problem Formulation}

Our objective is to evaluate and optimize the risk of the cumulative cost generated by a diffusion process.

On the filtered probability space $(\varOmega, \mathcal{F}, \mathbb{P}, \mathbb{F})$, we consider control processes $u: [0, T]\times\varOmega\to U$ such that $u(\cdot)$ is $\Fb$-adapted, where $U\subset\mathbb{R}^m$ is a compact set, and a diffusion process under any such control with initial time $t\in[0, T]$ and state $x\in \mathbb{R}^n$:
\begin{equation}
\label{s:cdp}
\left\{\begin{array}{ll}
\xdif X^{t, x; u}_s=b(s, X^{t, x; u}_s, u_s)\xdif s+\sigma(s, X^{t, x; u}_s, u_s)\xdif W_s,\quad s\in [t,T],\\
\quad X_t^{t, x; u}=x.\end{array}\right.
\end{equation}
Here,  $b: [0,T]\times \Rb^n\times U \to \Rb^n$ and $\sigma:  [0,T]\times \Rb^n\times U\to\Rb^{n\times d}$ are Borel measurable functions. We also introduce the \emph{cost rate} function: a measurable map $c: [0,T]\times\Rb^n \times U \to \Rb$, and the \emph{final state cost}: a measurable function $\varPsi: \Rb^n \to \Rb$. Therefore, the random cost accumulated on the interval $[t, T]$ for any $t\in[0, T]$ can
be expressed as follows:
\begin{equation}
\label{xiU}
\xi_{t,T}(u,x) : = \int_t^T\,c(s,X_s^{t,x; u},u_s)\xdif s + \varPsi(X^{t,x; u}_T),\quad \mbox{ a.s.. }
\end{equation}
\begin{assumption}
\label{s-assumption}
A constant $K>0$ exists such that, for any $s\in[t, T]$ and $(x_1, u_1),(x_2, u_2)\in\Rb^n\times U$,
the functions $b$, $\sigma$, $c$, and $\varPsi$  satisfy the following conditions:
\begin{gather*}
|b(s, x_1, u_1) - b(s, x_2, u_2)| + |\sigma(s, x_1, u_1)-\sigma(s, x_2, u_2)| + |c(s, x_1, u_1) - c(s, x_2, u_2) |\\
 \leq
 K \big(|x_1-x_2| + |u_1-u_2|\big),
 \\
|b(s, x_1, u_1)| + |\sigma(s, x_1, u_1)| + |c(s, x_1, u_1)| + |\varPsi(x_1)| \leq K(1+ |x_1| +|u_1|\big).
\end{gather*}
\end{assumption}
Under {Assumption} \ref{s-assumption}, the controlled diffusion process (\ref{s:cdp}) has a strong solution and
the cost functional is square integrable.

We define the \emph{control value function} as follows:
\begin{align}
\label{eq:control-value}
V^u(t, x):= \rho^g_{t, T}[\,\xi_{t, T}(u,x)\,] ,\quad \mbox{ a.s., }
\end{align}
where $\big\{\rho^g_{t,T}\big\}_{t\in [0,T]}$, is a system of $g$-evaluations discussed in section
\ref{s:BSDE}.
Using Definition \ref{d:gev}, we can express the control value function as follows:
\begin{align*}
V^u(t, x) &=  \xi_{t, T}(u, x)+\int^T_t\,g(s, Z^{t, x; u}_s)\xdif s -\int^T_t\,Z^{t, x; u}_s\xdif W_s\\
&= \varPsi(X^{t, x; u}_T) + \int_t^T\,\big[ c(s,X^{t, x; u}_s,u_s) + g(s, Z^{t, x; u}_s)\big] \xdif s -\int^T_t Z^{t, x; u}_s\xdif W_s,
\end{align*}
where
$(Y^{t, x; u},Z^{t, x; u})$ solve the following BSDE:
\begin{equation}
\label{s-fbsde}
\left\{\begin{array}{ll}
-\xdif Y^{t, x; u}_s = \big[ c(s, X^{t, x; u}_s,u_s) + g(s, Z^{t, x; u}_s)\big]\xdif s - Z^{t, x; u}_s\xdif W_s,\quad s\in[t, T],\\
Y^{t, x; u}_T = \varPsi(X^{t, x; u}_T).\end{array}\right.
\end{equation}
Equivalently, $V^u(t, x)= Y^{t, x; u}_t$.

If Assumptions \ref{s-assumption},  \ref{a:pp1}, and \ref{a:riskmeasure} are satisfied, then for every $(t,x) \in [0,T]\times \mathbb{R}^n$, the BSDE \eqref{s-fbsde} has a unique solution  $(Y^{t,x;u}, Z^{t,x;u})\in \Sc^{2}[t, T]\times \Hc^{2,d}[t, T]$ (see, Peng \cite{PSG1}), and, therefore, the control value function is well-defined.

In this way, the study of a risk-averse controlled system has been reduced to the study of controlled forward-backward stochastic differential equations (FBSDE).
Such systems ware extensively studied by Ma and Yong in \cite{MA}; other important references are \cite{AT,PT,Zhang,Yong}. In our case, the FBSDE is \emph{decoupled}, that is, the solution of the
backward equation does not affect the forward equation, which substantially simplifies the analysis
and allows for further advances.

Notice, when the driver $g\equiv 0$,  the control value function \eqref{eq:control-value} reduces to
the expected value of \eqref{xiU}. The risk-aversion is incorporated if other other drivers satisfying Assumption \ref{a:riskmeasure} are considered. By the comparison theorem of Peng \cite{PP2}, if $g_1$ is dominated by $g_2$, i.e., $g_1 \le g_2$, then $\rho_{t,T}^{g_1}(\xi_{t,T}(u,x)) \le \rho_{t,T}^{g_2}(\xi_{t,T}(u,x))$ almost surely; the larger the driver, the more risk aversion in the objective functional. For example, if we use $g_1(t,z) = \kappa|z|$, and $g_2(t,z) = \kappa |z_+|$, with $\kappa>0$, then $g_1$ dominates $g_2$.

\subsection{Risk-Averse Dynamic Programming and Hamilton--Jacobi--Bellman Equations}
We  now  proceed to the control problem. We define the {admissible control system} as  in Yong and Zhu \cite[p. 177]{XYZ}).
\begin{dfntn}\label{admissible-control} $\mathscr{U}[t, T]$ is called an \emph{admissible control system} if it satisfies the following conditions:
\begin{tightlist}{iv}
\item $(\varOmega, \mathcal{F}, \mathbb{P})$ is a complete probability space;
\item $\{W(s)\}_{s\geq t}$ is an $d$-dimensional standard Brownian motion defined on $(\varOmega, \mathcal{F}, \mathbb{P})$ over $[t, T]$ and $\mathbb{F}^t = (\mathcal{F}^t_s)_{s\in [t, T]}$, where $\mathcal{F}_s^t = \sigma\{\left(W_s; t \leq s \leq T\right)\cup\mathscr{N}\}$ and $\mathscr{N}$ is the collection of all $P$-null sets in $\mathcal{F}$;
\item $u: [s, T]\times\varOmega \to U$ is an $\{\mathcal{F}_s^t\}_{s\geq t}$-adapted process with $\mathbb{E}\int_t^T |u_s|^2\xdif s < +\infty$;
\item For any $x\in\mathbb{R}^n$ the system \eqref{s:cdp}--\eqref{s-fbsde}  admits a unique solution $(X, Y, Z)$ on $(\varOmega,\mathcal{F},\mathbb{P}, \mathbb{F}^t)$.
\end{tightlist}
\end{dfntn}

The \emph{optimal value function} $V:[0,T]\times\Rb^n\to\Rb$ is defined as follows:
\begin{equation}
\label{eq:valuefunction}
V(t,x) = \inf_{u\in\mathcal{U}[t, T]} V^{u}(t, x).
\end{equation}
The weak formulation of a risk-averse control problem is the following:
{given $(t, x)\in[0, T)\times\mathbb{R}^n$,
find ${u}^*\in\mathcal{U}[t, T]$
such that}
\begin{equation}
\label{w-formulation}
V^{{u}^*}(t, x) = \inf_{u\in\mathcal{U}[t, T]} V^{u}(t, x).
\end{equation}

We can now formulate the dynamic programming equation for our control problem.
\label{s:DP-HJB}
\begin{thrm}[\cite{ARYAO}]
\label{t:DP-equation}
Suppose Assumptions \Rref{s-assumption}, \Rref{a:pp1}, and \Rref{a:riskmeasure} are satisfied. Then, for any $(t, x)\in[0, T)\times\Rb^n$ and all $r\in [t,T]$, we have
\begin{equation}
\label{DPE}
V(t, x) = \inf_{u(\cdot)\in \mathcal{U}}  \rho_{t,r}^g\bigg[
\int_t^r c\big(s,X^{t, x;u}_s,u_s\big)\xdif s + V\big(r, X^{t, x;u}_r\big)\bigg].
\end{equation}
\end{thrm}

For $\alpha\in U$ we define  the \emph{Laplacian operator} $\Lb^\alpha$ as follows:
 for $w\in \Cc_{\textup{b}}^{1, 2}([0, T]\times\Rb^n)$ and $(t,x)\in [0,T]\times\Rb^n$,
\[
\big[ \Lb^\alpha w \big] (t,x) =
\partial_t w(t, x) +
\sum^n_{i,j = 1} \frac{1}{2} \big(\sigma(t, x,\alpha)\sigma(t, x,\alpha)^\top\big)_{ij} \partial_{x_i x_j} w(t,x) + \sum_{i=1}^n b_i(t,x,\alpha)\partial_{x_i} w(t,x).
\]
On the space $\Cc^{1,2}_{\textup{b}}([0,T]\times\Rb^n)$, we consider the following equation
\begin{equation}
\label{RHJB-v}
\min_{\alpha\in U}
\Big\{c(t, x,\alpha) +\big[\Lb^{\alpha } v\big](t, x) + g\big(t, [\mathcal{D}_x v\cdot\sigma^{\alpha}](t,x)\big)\Big\} = 0,
\quad (t,x)\in [0,T]\times\Rb^n,
\end{equation}
with the boundary condition
\begin{equation}
\label{RHJB-boundary}
v(T,x) = \varPsi(T,x),\quad x\in \Rb^n.
\end{equation}
We call \eqref{RHJB-v}--\eqref{RHJB-boundary} the \emph{risk-averse Hamilton--Jacobi--Bellman equation} associated with the
controlled system \eqref{s:cdp} and the risk functional \eqref{xiU}. It is
a generalization of the classical Hamilton--Jacobi--Bellman Equation with the extra term $g(\cdot,\cdot)$ responsible for risk aversion.
In the special case, when $g \equiv 0$, we obtain the standard equation.

The following two theorems can be derived from general results on fully coupled forward--backward systems
in \cite{JLQW}. For decoupled systems, a direct proof is provided in \cite{ARYAO}.
\begin{thrm}
\label{t:RHJB-v}
Suppose Assumptions \Rref{s-assumption},  \Rref{a:pp1}, and \Rref{a:riskmeasure} are satisfied; in addition, the functions  $b$ and $\sigma$ are bounded in $x$.
Then the value function $V(\cdot,\cdot)$ is a viscosity solution of the equation \eqref{RHJB-v}--\eqref{RHJB-boundary}.
\end{thrm}

It is clear that if $V\in \Cc^{1,2}_{\textup{b}}([t,T]\times\Rb^n)$ then it satisfies \eqref{RHJB-v}--\eqref{RHJB-boundary}.
We can also prove the converse relation (\emph{verification theorem}).

\begin{thrm}
Suppose the assumptions of Theorem \Rref{t:RHJB-v} are fulfilled and let the function $K \in \Cc^{1,2}_{\textup{b}}([t,T]\times\Rb^n)$ satisfy  \eqref{RHJB-v}--\eqref{RHJB-boundary}.
Then $K(t, x) \leq V^u(t, x)$ for any control $u(\cdot)\in\mathcal{U}$  and all $(t,x)\in [0,T]\times \Rb^n$.
Furthermore, if a control process
$u^*(\cdot)\in \mathcal{U}$ exists, satisfying for almost all $(s,\omega)\in [0, T]\times \varOmega$,
together with the corresponding trajectory $X^{0, x;u^*}_s$, the relation
\begin{equation}
\label{u-star}
u^*_s \in \argmin_{\alpha\in U} \Big\{c(s, X^{0, x;u^*}_s,\alpha) + \Lb^{\alpha} K(s, X^{0, x;u^*}_s)
+ g\big(t, [\mathcal{D}_x K\cdot \sigma^{\alpha}](t,X^{0, x;u^*}_s)\big)\Big\},
\end{equation}
then $K(t, x) = V(t, x) = V^{u^*}(t,x)$ for all $(t,x)\in [0,T]\times\Rb^n$.
\end{thrm}

\section{Piecewise-Constant Control Policies and the Perturbed Problem}
\label{s:perturbed}
Let $h^2\in(0, 1]$ be a time discretization step. We use the square to simplify further analysis.

\begin{dfntn}\label{admissible-control-h} For any $h\in(0, 1]$ and $t\in [0,T)$, let $\mathcal{U}_h^t$ be the subset of $\mathcal{U}$ consisting of all $\Fb$-adapted processes $u_t$ which are constant on intervals $[t,t+ h^2)$, $[t+h^2,t+ 2h^2)$, \dots, $[t+kh^2,T]$, where $T-h^2 \le t+ kh^2 \le T$.
\end{dfntn}

We define the corresponding value function $V_h:[0,T]\times\Rb^n\to\Rb$ as follows:
\begin{equation}
\label{optimal-value-h}
V_h(t,x) = \inf_{u(\cdot)\in\mathcal{U}_h^t} V^{u}(t, x).
\end{equation}

We assume a stronger condition than Assumptions \ref{a:pp1} and \ref{s-assumption}.

\begin{assumption}\label{ass:sde} Let $\mu(t, x, z, \alpha)$ stand for $\sigma(t, x, \alpha)$, $b(t,x, \alpha)$, $c(t, x, \alpha)$, and $\varPhi(x)$ \footnote{We sometimes write $\mu^{\alpha}$ instead of $\mu(\cdot,\cdot,\cdot,\alpha)$}. We assume that a constant $K$ exists such that
\begin{tightlist}{ii}
\item For all  $\alpha, \alpha_1, \alpha_2\in U$, $x, x_1, x_2\in\Rb^n$, $z, z_1, z_2\in\Rb^d$ we have
\begin{gather*}
|\mu(t, x, z, \alpha)|\leq K, \\
|\mu(t, x_1, z_1, \alpha_1) - \mu(t, y, z_2, \alpha_2)|\leq K\left(|x_1 - x_2| + |z_1 - z_2| + |\alpha_1 - \alpha_2|\right);
\end{gather*}
\item For all $\alpha\in U$, $s,t\in [0, T]$,  $x\in\Rb^n$, $z\in\Rb^d$, we have
\begin{align*}
|\mu(t, x, z, \alpha) - \mu(s, x, z, \alpha)|\leq K|t-s|^{1/2}.
\end{align*}
\end{tightlist}
\end{assumption}
By general results on forward--backward systems,
the system \eqref{s:cdp}, \eqref{xiU} and \eqref{s-fbsde} has a unique solution  and thus both functions:
$V$  in \eqref{eq:valuefunction} and $V_h$  in \eqref{optimal-value-h}, are well-defined. In particular, they are both deterministic. We focus on the difference between the value functions $V$ and $V_h$.

The idea is is to embed the original control problem into a family of time and space perturbed problems, and then obtain a smooth approximation of the value function by means of an integral regularization (mollification).

Let  $B = \{(\tau, \zeta)\in \Rb\times\Rb^n: \tau\in(-1, 0),\; |\zeta| <1 \}$.
 Consider a time $t\in [0,T]$ and
time instants $t_i=t+ih^2$, $i=0,1,\dots,k$ and $t_{k+1}=T$. For a piecewise-constant control
$u_s = \alpha_i$,  $s\in [t_i, t_{i+1})$, $i=0,1,\dots,k$, and
perturbations $\beta_i\in B$, $i= 0,1,\dots,k$, we define the perturbed controlled  FBSDE system:
\begin{align}
\xdif\widetilde{X}_s  &=   b(s + \varepsilon^2\tau_i, \widetilde{X}_s+\varepsilon\zeta_i, \alpha_i)\xdif s
+  \sigma(s+\varepsilon^2\tau_i, \widetilde{X}_s +\varepsilon\zeta_i, \alpha_i)\,\xdif\widetilde{W}_s, \label{tildeX}\\
\xdif\widetilde{Y}_s &=
 \big[ c(s+\varepsilon^2\tau_i, \widetilde{X}_s+\varepsilon\zeta_i,\alpha_i)
+ g(s+\varepsilon^2\tau_i, \tilde{Z}_s)\big] \xdif s
- \tilde{Z}_s\,\xdif\widetilde{W}_s,  \label{tildeY}\\
&\qquad s\in [t_i,t_{i+1}),\quad i=0,1,\dots,k,\notag
\end{align}
with a fixed $\varepsilon > 0$, with the initial condition $\widetilde{X}_t=x$, and with the final condition $\widetilde{Y}_T=\varPhi(\widetilde{X}_T)$. The process $\widetilde{W}$ is a Brownian motion. We assume here that $b(t,x,\alpha) = b(0,x,\alpha)$, $\sigma(t,x,\alpha) = \sigma(0,x,\alpha)$, $c(t,x,\alpha) = c(0,x,\alpha)$,
and $g(t,z)= g(0,z)$ for all $t\in [-\varepsilon^2,0]$.

We consider the following discrete-time optimal control problem associated with the system \eqref{tildeX}--\eqref{tildeY}.
At each time~$t_i$, we select a control value $\alpha_i\in U$ and a perturbation $\beta_i\in B$. The system evolves to time $t_{i+1}$, when new controls
$\alpha_{i+1}$ and $\beta_{i+1}$ are selected. The objective of the controller is to make $\widetilde{Y}_t$ the smallest possible. From now on, we use $\bar{\alpha}$ and $\bar{\beta}$
to represent the random sequences $\alpha_i$ and $\beta_i$, for $i=0,1\dots,k$.

\begin{lmm}
\label{l:shift}
Functions $\widetilde{V}^{\bar{\alpha},\bar{\beta}}:\{t_0,t_1,\dots,t_{k}\}\times\Rb^n\to\Rb$ exist, such that for all $x_i\in \Rb^n$, if the
system \eqref{tildeX}--\eqref{tildeY} starts at time $t_i$ from $\widetilde{X}_{t_i}=x_i$, then
$\widetilde{Y}_{t_i}=\widetilde{V}^{\bar{\alpha},\bar{\beta}}(t_i,x_i)$. Moreover,
\begin{multline}
\label{tildeV}
\widetilde{V}^{\bar{\alpha},\bar{\beta}}(t_{i},x_i) = \rho^g_{t_i+\varepsilon^2\tau_i,{t_{i+1}+\varepsilon^2\tau_i}}\bigg[
\int_{t_i+\varepsilon^2\tau_i}^{t_{i+1}+\varepsilon^2\tau_i} \hspace{-0.5em}c\big(s, X_{s}^{t_i+\varepsilon^2\tau_i, x_i+\varepsilon\zeta_i; \alpha_i}, \alpha_i\big)\xdif s \\
 {} + \widetilde{V}^{\bar{\alpha},\bar{\beta}}\big(t_{i+1},X_{t_{i+1}+\varepsilon^2\tau_i}^{t_i+ \varepsilon^2\tau_i,x_i +\varepsilon\zeta_i;\alpha_i}-\varepsilon\zeta_i\big) \bigg].
\end{multline}
\end{lmm}
\begin{proof}
With $\widetilde{W}_s \sim W_{s+\varepsilon^2\tau}$ for $s\in[t_i,t_{i+1}]$
directly from the equations  \eqref{tildeX} and \eqref{s:cdp}
we obtain:
\begin{equation}
\label{shift}
 \widetilde{X}_{s}^{t_i, x_i;\alpha_i}= X_{s+\varepsilon^2\tau_i}^{t_i+ \varepsilon^2\tau_i,x_i +\varepsilon\zeta_i;\alpha_i}-\varepsilon\zeta_i,
\quad s\in [t_i,t_{i+1}],\quad\text{a.s.}
\end{equation}
With this substitution, the  BSDE \eqref{tildeY} at $s=t_i$ is equivalent to:
\begin{align*}
\widetilde{Y}_{t_i} &=
 \widetilde{Y}_{t_{i+1}}
 + \int_{t_i+\varepsilon^2\tau_i}^{t_{i+1}+\varepsilon^2\tau_i} \hspace{-0.2em}\big[c\big(s, X_{s}^{t_i+\varepsilon^2\tau_i, x_i+\varepsilon\zeta_i; \alpha_i}, \alpha_i\big)+g(s, {Z}_s)\big]\xdif s
- \int_{t_i+\varepsilon^2\tau_i}^{t_{i+1}+\varepsilon^2\tau_i}\hspace{-0.5em}{Z}_s\xdif W_s\\
& = \rho^g_{t_i+\varepsilon^2\tau_i,{t_{i+1}+\varepsilon^2\tau_i}}\left[
\int_{t_i+\varepsilon^2\tau_i}^{t_{i+1}+\varepsilon^2\tau_i} \hspace{-0.5em}c\big(s, X_{s}^{t_i+\varepsilon^2\tau_i, x_i+\varepsilon\zeta_i; \alpha_i}, \alpha_i\big)\xdif s + \widetilde{Y}_{t_{i+1}} \right].
\end{align*}
By definition, $\widetilde{Y}_{T}=\varPhi(\widetilde{X}_T)$. Supposing $\widetilde{Y}_{t_{j+1}}=\widetilde{V}^{\bar{\alpha},\bar{\beta}}\big(t_{j+1},\widetilde{X}_{t_{j+1}}\big)$ for some $j$, proceeding backwards in time
we conclude from the last equation
that we can write $\widetilde{Y}_{t_{i}}=\widetilde{V}^{\bar{\alpha},\bar{\beta}}(t_{i},x_i)$ for some function $\widetilde{V}^{\bar{\alpha},\bar{\beta}}(\cdot,\cdot)$.
We can thus write the recursive relation
\[
\widetilde{V}^{\bar{\alpha},\bar{\beta}}(t_{i},x_i) =
\rho^g_{t_i+\varepsilon^2\tau_i,{t_{i+1}+\varepsilon^2\tau_i}}\left[
\int_{t_i+\varepsilon^2\tau_i}^{t_{i+1}+\varepsilon^2\tau_i} \hspace{-0.5em}c\big(s, X_{s}^{t_i+\varepsilon^2\tau_i, x_i+\varepsilon\zeta_i; \alpha_i}, \alpha_i\big)\xdif s + \widetilde{V}^{\bar{\alpha},\bar{\beta}}\big(t_{i+1},\widetilde{X}_{t_{i+1}}\big) \right].
\]
Substitution of \eqref{shift} proves \eqref{tildeV}.
\end{proof}

Using this recursive relation, we define the value function of the optimally perturbed problem $\widetilde{V}_{h,\varepsilon}(t_i,x_i)$ at each time $t_i$
and the corresponding state $x_i$ as follows.
At $t_{k+1}=T$ we set $\widetilde{V}_{h,\varepsilon}(T,x_T) = \varPhi(x_T)$, and then, proceeding backwards in time,
\begin{multline*}
\widetilde{V}_{h,\varepsilon}(t_i,x_{t_i}) = \inf_{\alpha_i\in U}\inf_{\beta_i\in B}
\rho^g_{t_i+\varepsilon^2\tau_i,{t_{i+1}+\varepsilon^2\tau_i}}\bigg[
\int_{t_i+\varepsilon^2\tau_i}^{t_{i+1}+\varepsilon^2\tau_i} \hspace{-0.5em}c\big(s, X_{s}^{t_i+\varepsilon^2\tau_i, x_i+\varepsilon\zeta_i; \alpha_i}, \alpha_i\big)\xdif s\\
 + \widetilde{V}_{h,\varepsilon}\big(t_{i+1},X_{t_{i+1}+\varepsilon^2\tau_i}^{t_i+ \varepsilon^2\tau_i,x_i +\varepsilon\zeta_i;\alpha_i}-\varepsilon\zeta_i\big) \bigg].
\end{multline*}
This construction can be carried out for every $t\in [0,T]$ and the resulting points $t_i=t+ih^2$, thus defining a function
$\widetilde{V}_{h,\varepsilon}:[0,T]\times\Rb^n\to\Rb$ which satisfies the relation
\begin{multline}
\label{V-recursive}
\widetilde{V}_{h,\varepsilon}(t,x) = \inf_{\alpha_i\in U}\inf_{\beta_i\in B}
\rho^g_{t+\varepsilon^2\tau,{t+h^2+\varepsilon^2\tau}}\bigg[
\int_{t+\varepsilon^2\tau}^{t+h^2+\varepsilon^2\tau} \hspace{-0.5em}c\big(s, X_{s}^{t+\varepsilon^2\tau, x+\varepsilon\zeta; \alpha_i}, \alpha_i\big)\xdif s \\
{} + \widetilde{V}_{h,\varepsilon}\big(t+h^2,X_{t+h^2+\varepsilon^2\tau}^{t+ \varepsilon^2\tau,x +\varepsilon\zeta;\alpha_i}-\varepsilon\zeta\big) \bigg].\qquad
\end{multline}
If $t\in (T-h^2,T)$ we replace $t+h^2$ with $T$ in the above equation. The function $\widetilde{V}_{h,\varepsilon}(t,x)$ represents the optimal value of the perturbed problem
starting at time $t$ from the state $x$ and proceeding with piecewise constant controls and perturbations on intervals of length $h^2$ (except, perhaps, the last one, which ends at $T$). Let us stress that the perturbations are treated as additional controls in this construction.

We now present a number of useful estimates from Krylov \cite{Krylov1}.
\begin{lmm}\label{lem:krylov} For $t\in [0,T)$, $x, y\in\mathbb{R}^n$,
 and $\bar{\alpha}\in \mathcal{U}^t_h$, denote by $\widetilde{X}_s^{t,x;\bar{\alpha},\bar{\beta}}$ the solution of \eqref{tildeX} and by $X_s^{t,x;\bar{\alpha}}$  the solution of \eqref{s:cdp}, with the initial state $x\in\mathbb{R}^n$  at time~$t$. Then
\begin{align}
\label{eq:estimateskrylov}
&\mathbb{E}\Big[\,\sup_{t\le s\leq T}\big|\widetilde{X}_s^{t,x;\bar{\alpha}, \bar{\beta}} - X_t^{s, x; \bar{\alpha}}\big|^2\,\Big]\leq Ne^{NT}\varepsilon^2,\\
&\mathbb{E}\Big[\,\sup_{t \le s\leq T}\big|\widetilde{X}_s^{t,x;\bar{\alpha}, \bar{\beta}}-\widetilde{X}_s^{t,y; \bar{\alpha},\bar{\beta}}\big|^2\,\Big]\leq Ne^{NT}|x-y|^2,\\
&\mathbb{E}\Big[\,\sup_{t\leq r\leq T} \big|\widetilde{X}_s^{t,x;\bar{\alpha}, \bar{\beta}}-\widetilde{X}_s^{t,y;\bar{\alpha}, \bar{\beta}}\big|^2\,\Big]\leq Ne^{NT}|t-r|,
\end{align}
for $N>0$ depending on $(K, d, n)$ only.
\end{lmm}

The proof is by using Theorem 2.5.9 in \cite{Krylov}. These estimates can be used to derive
the following bounds.

\begin{lmm}\label{cor:be}
A constant $N$  exists, depending on ($K$,  $d$, $n$) only, such that:\vspace{1ex}
\begin{tightlist}{ii}
\item For $t\in[0, T]$ and $x\in\Rb^d$, we have
$
\big|\widetilde{V}_{h,\varepsilon}(t, x) - V_h(t, x)\big|  \leq Ne^{NT}\varepsilon$,
\item For  $t, r\in[0, T]$, and $x, y\in\Rb^n$, we have
$
 \big|\widetilde{V}_{h,\varepsilon}(t, x) - \widetilde{V}_{h,\varepsilon}(r, y)\big|  \leq Ne^{NT}(|x-y|+|t-r|^{\frac{1}{2}})$.
\end{tightlist}
\end{lmm}
\begin{proof}
For fixed $\bar{\alpha}, \bar{\beta}$, recall that
\begin{align*}
V^{\bar{\alpha}}(t, x) &= \rho^g_{t, T}\bigg[\int_t^T\, c(r, X_r^{s, x; \bar{\alpha}}, \bar{\alpha}_r)\xdif r +\varPhi\big(X_T^{t,x;\bar{\alpha}}\big)\bigg], \\
\widetilde{V}_{h,\varepsilon}^{\bar{\alpha}, \bar{\beta}}(t, x) &= \rho^g_{t, T}\bigg[\,\int_t^T\,c(r, \widetilde{X}_r^{t,x;\bar{\alpha},\bar{\beta}}, \bar{\alpha}_r)\xdif r + \varPhi\big(X_T^{t,x;\bar{\alpha},\bar{\beta}}\big)\bigg].
\end{align*}
By standard estimates for BSDE and Lemma \ref{lem:krylov}, we have  with some $\gamma>0$ depending on $(K,d ,n)$,
\begin{align*}
\lefteqn{|V^{\bar{\alpha}}(t, x) - \widetilde{V}_{h,\varepsilon}^{\bar{\alpha}, \bar{\beta}}(t, x)|
 \leq  \mathbb{E}\Big[\,\big|\varPhi(X_T^{t,x;\bar{\alpha}}) - \varPhi(\widetilde{X}_T^{t,x;\bar{\alpha},\bar{\beta}})\big|^2e^{\gamma (T-t)}\,\Big]}\\
  & {\qquad } + \mathbb{E}\Big[\,\int_t^T\, \big|c(r, X_r^{t,x;\bar{\alpha}}, \bar{\alpha}_r) - c(r, \widetilde{X}_r^{t,x;\bar{\alpha},\bar{\beta}}, \bar{\alpha}_r)\big|^2e^{\gamma(r-t)}\xdif r\,\Big]\\
 & \leq ~K^2e^{\gamma(T-t)}\mathbb{E}\Big[\,\big|X_r^{t,x;\bar{\alpha}} - \widetilde{X}_r^{t,x;\bar{\alpha},\bar{\beta}}\big|^2\,\Big] + KT^2e^{\gamma(T-t)}\mathbb{E}\Big[\,\sup_{t\leq T}\big|X_r^{t,x;\bar{\alpha}} - \widetilde{X}_r^{t,x;\bar{\alpha},\bar{\beta}}\,\big|^2\,\Big]
\leq  Ne^{NT}\varepsilon.
\end{align*}
The first assertion follows.
For (ii), we observe that
\begin{align*}
\big|\widetilde{V}_{h,\varepsilon}^{\bar{\alpha}, \bar{\beta}}(t,x) - \widetilde{V}_{h,\varepsilon}^{\bar{\alpha}, \bar{\beta}}(r, y)\big| \leq \big|\widetilde{V}_{h,\varepsilon}^{\bar{\alpha}, \bar{\beta}}(t,x) - \widetilde{V}_{h,\varepsilon}^{\bar{\alpha}, \bar{\beta}}(t, y)\big| + \big|\widetilde{V}_{h,\varepsilon}^{\bar{\alpha}, \bar{\beta}}(t_,y) - \widetilde{V}_{h,\varepsilon}^{\bar{\alpha}, \bar{\beta}}(r, y)\big|.
\end{align*}
Similar to the proof for (i), by applying the second and third inequalities of Lemma \ref{lem:krylov}, we have
\begin{align*}
&\big|\widetilde{V}_{h,\varepsilon}^{\bar{\alpha}, \bar{\beta}}(t,x) - \widetilde{V}_{h,\varepsilon}^{\bar{\alpha}, \bar{\beta}}(t, y)\big| \leq Ne^{NT}|x-y|^2, 
\\
& \big|\widetilde{V}_{h,\varepsilon}^{\bar{\alpha}, \bar{\beta}}(t_,y) - \widetilde{V}_{h,\varepsilon}^{\bar{\alpha}, \bar{\beta}}(r, y)\big|\leq Ne^{NT}|t-r|^\frac{1}{2},
\end{align*}
which implies the postulated estimates.
\end{proof}

\section{Mollification of the Value Function}
\label{s:mollification}
We now introduce an integral transformation of the value function. We take
a non-negative function $\varphi\in \Cc^{\infty}(B)$ with \mbox{$\int_{B}\varphi(\tau, \zeta) \xdif\tau\xdif \zeta= 1$}, called a \emph{mollifier}.
For $\varepsilon > 0$, we re-scale the {mollifier} as
$\varphi_{\varepsilon}(\tau, \zeta) = \varepsilon^{-n-2}\varphi(\tau/\varepsilon^2, \zeta/\varepsilon)$,
and we introduce the following notation of the convolution of the function $\widetilde{V}_{h,\varepsilon}$ with the re-scaled mollifier:
\[
\widehat{V}_{h,\varepsilon}(t, x) = \big[ \widetilde{V}_{h,\varepsilon}\star \varphi_{\varepsilon}\big] (t, x) = \int_B \widetilde{V}_{h,\varepsilon}(t-\varepsilon^2\tau,x-\varepsilon \zeta) \varphi(\tau, \zeta)\xdif \tau\xdif \zeta,
\]
where $t\in [0,T-\varepsilon^2]$ and $x\in \Rb^n$.
We shall need an estimate of the seminorm  $\big\| \widehat{V}_{h,\varepsilon}\big\|_{2,1}$ defined as follows:
\begin{multline*}
\big\| w \big\|_{2,1} =
\sup_{(t,x)}\big| w(t,x)\big| + \sup_{(t,x)}\big\|\Dc_x w(t,x)\big\|
+ \sup_{(t,x)}\big\|\Dc^2_{xx}w(t,x)\big\|
+ \sup_{(t,x)}\big|\partial_t w(t,x)\big| \\
{} +
 \sup_{(t,x),(s,y)}\frac{\big\|\Dc^2_{xx} w(t,x) - \Dc^2_{xx} w(s,y)\big\|}{|t-s|+|x-y|}
 +
 \sup_{(t,x),(s,y)}\frac{\big|\partial_t w(t,x) - \partial_t w(s,y)\big|}{|t-s|+|x-y|}.
\end{multline*}
In the formula above, we use $\Dc_x$ and $\Dc_{xx}^2$ to denote the gradient and the Hessian matrix, and the
supremum is always over $(t,x),(s,y)\in [0,T-\varepsilon^2]\times \Rb^2$.

\begin{lmm}\label{estimatesall} If $\varepsilon\geq h$, then
$\Big\|\widehat{V}_{h,\varepsilon}\Big\|_{2,1}\leq Ne^{NT}\varepsilon^{-2}$ and
$\Big|\widehat{V}_{h,\varepsilon}-\widetilde{V}_{h,\varepsilon}\Big|_0\leq Ne^{NT}\varepsilon$.
\end{lmm}
\begin{proof}
By elementary properties of the convolution,
\begin{multline*}
\frac{\partial}{\partial t}\widehat{V}_{h,\varepsilon}(t, x)=\frac{\partial}{\partial t}\left(\widetilde{V}_{h,\varepsilon}*\varphi_{\varepsilon}\right)(t, x) =
\left(\widetilde{V}_{h,\varepsilon}*\frac{\partial}{\partial t}\varphi_{\varepsilon}\right)(t, x) \\
=
\varepsilon^{-2}\int_{B} \widetilde{V}_{h,\varepsilon}(t-\varepsilon^2\tau, x-\varepsilon\zeta)\frac{\partial}{\partial t}\varphi(\tau, \zeta)\xdif \tau \xdif \zeta.
\end{multline*}
Thus, due to Lemma \ref{cor:be}(ii),
\begin{align*}
\Big|\frac{\partial}{\partial t}\widehat{V}_{h,\varepsilon}(t, x)\Big| & = \varepsilon^{-2}\Big|\int_{B} \widetilde{V}_{h,\varepsilon}(t-\varepsilon^2\tau, x-\varepsilon\zeta)\frac{\partial}{\partial t}\varphi(\tau, \zeta)\xdif \tau \xdif \zeta\Big|\\
& = \varepsilon^{-2}\Big|\int_{B}\big[\widetilde{V}_{h,\varepsilon}(t-\varepsilon^2\tau, x-\varepsilon\zeta) - \widetilde{V}_{h,\varepsilon}(t, x)\big]\frac{\partial}{\partial t}\varphi(\tau, \zeta)\xdif \tau\xdif\zeta\Big| \\
& \le 2 Ne^{NT}\varepsilon^{-1} \int_{B}\Big|\frac{\partial}{\partial t}\varphi(\tau, \zeta)\Big| \xdif \tau\xdif\zeta.
\end{align*}
We can thus increase $N$ to write for all $\varepsilon>0$ the inequality
\[
\Big|\frac{\partial}{\partial t}\widehat{V}_{h,\varepsilon}(t, x)\Big| \le Ne^{NT}\varepsilon^{-1}.
\]
Similarly, after redefining $N$ in an appropriate way,
\begin{align*}
&\Big|\frac{\partial}{\partial x^i}\widehat{V}_{h,\varepsilon}\Big|_0 \leq Ne^{NT}\leq Ne^{NT}\varepsilon^{-1},\qquad\Big\|\frac{\partial^2}{\partial x^i\partial x^j}\widehat{V}_{h,\varepsilon}\Big\|_0 \leq Ne^{NT}\varepsilon^{-1},\\
& \Big|\frac{\partial^2}{\partial t^2}\widehat{V}_{h,\varepsilon}\Big|_0 + \Big\|\frac{\partial^3}{\partial x^i\partial x^j\partial t}\widehat{V}_{h,\varepsilon}\Big\|_0 \leq Ne^{NT}\varepsilon^{-3},\\
& \Big\|\frac{\partial^2}{\partial t \partial x^i}\widehat{V}_{h,\varepsilon}\Big\|_0 + \Big\|\frac{\partial^3}{\partial x^i\partial x^j\partial x^k}\widehat{V}_{h,\varepsilon}\Big\|_0 \leq Ne^{NT}\varepsilon^{-2}.
\end{align*}
It follows that
\begin{align*}
\Big|\frac{\partial}{\partial t}\widehat{V}_{h,\varepsilon}(t, x) - \frac{\partial}{\partial t}\widehat{V}_{h,\varepsilon}(s, y)\Big| &\leq |t-s|\big|\frac{\partial^2}{\partial t^2}\widehat{V}_{h,\varepsilon}\Big|_0 + N|x-y|\Big\|\frac{\partial^2}{\partial x^i\partial t}\widehat{V}_{h,\varepsilon}\Big\|_0\\
&\leq Ne^{NT}|t-s|\varepsilon^{-3} + Ne^{NT}|x-y|\varepsilon^{-2}\\
&=  Ne^{NT}\varepsilon^{-2} \big( |t-s|\varepsilon^{-1}+|x-y|\big) .
\end{align*}
The last expression is less than $Ne^{NT}\varepsilon^{-2}\big(|t-s|^{\frac{1}{2}} + |x-y|\big)$ if $|t-s|\leq \varepsilon^2$. On the other hand, if $|t-s|\geq \varepsilon^2$, then
\[
\Big|\frac{\partial}{\partial t}\widehat{V}_{h,\varepsilon}(t, x) - \frac{\partial}{\partial t}\widehat{V}_{h,\varepsilon}(s, y)\Big| \leq 2\Big|\frac{\partial }{\partial t}\widehat{V}_{h,\varepsilon}\Big|_0
\leq Ne^{NT}\varepsilon^{-1}\leq Ne^{NT}\varepsilon^{-2}\big(|t-s|^{\frac{1}{2}} + |x-y|\big).
\]
Hence the inequality in question holds for all $t, s\in [0,T]$. In the same way one gets that
\begin{align*}
\Big\|\frac{\partial^2}{\partial x^i\partial x^j}\widehat{V}_{h,\varepsilon}(t, x) - \frac{\partial^2}{\partial x^i\partial x^j}\widehat{V}_{h,\varepsilon}(s, y)\Big\| \leq Ne^{NT}\varepsilon^{-2}\big(|t-s|^{\frac{1}{2}} + |x-y|\big).
\end{align*}
This proves the first inequality in the assertions.

To prove the second one, we notice that Lemma \ref{cor:be} yields
\begin{align*}
\big|\widetilde{V}_{h,\varepsilon}(t, x) - \widehat{V}_{h,\varepsilon}(t, x)\big|\leq \int_{B}\big|\widetilde{V}_{h,\varepsilon}(t, x)-\widetilde{V}_{h,\varepsilon}(t-\varepsilon^2\tau, x-\varepsilon\zeta)\big|\varphi(\tau, \zeta)\xdif \tau \xdif\zeta\leq  2Ne^{NT}\varepsilon.
\end{align*}
We can thus adjust $N$, if needed, to establish the second estimate for all $\varepsilon>0$.
\end{proof}

We can now establish a dynamic programming bound for the mollified value function.

\begin{lmm}\label{lem:mp1} Suppose Assumptions \Rref{a:riskmeasure} and \Rref{ass:sde} are satisfied. Then for all $x\in\Rb^n$, $t\in[0,T-\varepsilon^2-h^2]$, and all $\alpha\in U$ we have
\begin{equation}
\label{V-moll}
\widehat{V}_{h,\varepsilon}\big(t, x\big) \le \rho^g_{t, t+h^2}
\left[ \int_t^{t+h^2} \hspace{-0.5em}c(s, X_s^{t, x;u}, \alpha)\xdif s + \widehat{V}_{h,\varepsilon}\big(t+h^2 , X_{t+h^2}^{t, x;\alpha}\big) \right] + Ne^{NT} h^2\varepsilon,
\end{equation}
where $N$ is a constant independent on $h$, $\varepsilon$, and $T$.
\end{lmm}
\begin{proof}
Fixing $\beta=(\tau,\zeta)$ on the right hand side of \eqref{V-recursive}, for every $\alpha\in U$ we obtain the inequality
\[
\widetilde{V}_{h,\varepsilon}(t,x) \le
\rho^g_{t+\varepsilon^2\tau,{t+h^2+\varepsilon^2\tau}}\bigg[
\int_{t+\varepsilon^2\tau}^{t+h^2+\varepsilon^2\tau}\hspace{-0.5em} c\big(s, X_{s}^{t+\varepsilon^2\tau, x+\varepsilon\zeta; \alpha}, \alpha\big)\xdif s
+ \widetilde{V}_{h,\varepsilon}\big(t+h^2,X_{t+h^2+\varepsilon^2\tau}^{t+ \varepsilon^2\tau,x +\varepsilon\zeta;\alpha}-\varepsilon\zeta\big) \bigg].\qquad
\]
 Since $t\le T-\varepsilon^2-h^2$, we can substitute $t-\varepsilon^2\tau$ for $t$ and $x-\varepsilon\zeta$ for $x$. We obtain
 \[
\widetilde{V}_{h,\varepsilon}\big(t-\varepsilon^2\tau, x-\varepsilon\zeta\big)
\le \rho^g_{t, t+h^2}
\left[ \int_t^{t+h^2} \hspace{-0.5em} c(s, X_s^{t, x;u}, \alpha)\xdif s + \widetilde{V}_{h,\varepsilon}\big(t+h^2 -\varepsilon^2\tau, X_{t+h^2}^{t, x;\alpha} -\varepsilon\zeta\big) \right].
\]
By virtue of Theorem \ref{t:geval-prop}, the risk measure $\rho_{t, t+h^2}[\cdot]$ is subadditive, and thus
 \begin{equation}
 \label{V-pre-moll}
 \begin{aligned}
\widetilde{V}_{h,\varepsilon}\big(t-\varepsilon^2\tau, x-\varepsilon\zeta\big)
&\le \rho^g_{t, t+h^2}
\left[ \int_t^{t+h^2} \hspace{-0.5em} c(s, X_s^{t, x;u}, \alpha)\xdif s + \widehat{V}_{h,\varepsilon}\big(t+h^2, X_{t+h^2}^{t, x;\alpha}\big) \right] \\
& {\quad}+ \rho^g_{t, t+h^2}
\left[\widetilde{V}_{h,\varepsilon}\big(t+h^2 -\varepsilon^2\tau, X_{t+h^2}^{t, x;\alpha} -\varepsilon\zeta\big) - \widehat{V}_{h,\varepsilon}\big(t+h^2, X_{t+h^2}^{t, x;\alpha}\big)\right].
\end{aligned}
\end{equation}
The last term on the right hand side of \eqref{V-pre-moll} can be equivalently bounded by using the dual representation of the risk measure $\rho_{t, t+h^2}[\cdot]$. We can thus write the following chain of relations:
\begin{align*}
\lefteqn{ \rho^g_{t, t+h^2}
\left[ \widetilde{V}_{h,\varepsilon}\big(t+h^2 -\varepsilon^2\tau, X_{t+h^2}^{t, x;\alpha} -\varepsilon\zeta\big) - \widehat{V}_{h,\varepsilon}\big(t+h^2, X_{t+h^2}^{t, x;\alpha}\big)\right]}\quad \\
& = \sup_{\varGamma\in \mathcal{A}_{t, t+h^2}}\hspace{-0.5em} \Eb_{t}\left[\varGamma\left(  \widetilde{V}_{h,\varepsilon}\big(t+h^2 -\varepsilon^2\tau, X_{t+h^2}^{t, x;\alpha} -\varepsilon\zeta\big) - \widehat{V}_{h,\varepsilon}\big(t+h^2, X_{t+h^2}^{t, x;\alpha}\big)\right)\right]\\
& =  \Eb_{t}\left[ \widetilde{V}_{h,\varepsilon}\big(t+h^2 -\varepsilon^2\tau, X_{t+h^2}^{t, x;\alpha} -\varepsilon\zeta\big) - \widehat{V}_{h,\varepsilon}\big(t+h^2, X_{t+h^2}^{t, x;\alpha}\big)\right]\\
&{\ }+ \sup_{\varGamma\in \mathcal{A}_{t, t+h^2}} \hspace{-0.5em}\Eb_{t}\left[(\varGamma-1)\left(  \widetilde{V}_{h,\varepsilon}\big(t+h^2 -\varepsilon^2\tau, X_{t+h^2}^{t, x;\alpha} -\varepsilon\zeta\big)\right) - \widehat{V}_{h,\varepsilon}\big(t+h^2, X_{t+h^2}^{t, x;\alpha}\big)\right]\\
& \le  \Eb_{t}\left[ \widetilde{V}_{h,\varepsilon}\big(t+h^2 -\varepsilon^2\tau, X_{t+h^2}^{t, x;\alpha} -\varepsilon\zeta\big) - \widehat{V}_{h,\varepsilon}\big(t+h^2, X_{t+h^2}^{t, x;\alpha}\big)\right]\\
&{\ }+ \sup_{\varGamma\in \mathcal{A}_{t, t+h^2}} \hspace{-0.5em}\big\| \varGamma-1 \big\| \left\|  \widetilde{V}_{h,\varepsilon}\big(t+h^2 -\varepsilon^2\tau, X_{t+h^2}^{t, x;\alpha} -\varepsilon\zeta\big) - \widehat{V}_{h,\varepsilon}\big(t+h^2, X_{t+h^2}^{t, x;\alpha}\big)\right\|.
\end{align*}
Owing to Corollary \ref{c:Gamma-bound} and Lemma \ref{estimatesall}(ii), we obtain the estimate
\begin{align*}
\lefteqn{ \rho^g_{t, t+h^2}
\left[ \widetilde{V}_{h,\varepsilon}\big(t+h^2 -\varepsilon^2\tau, X_{t+h^2}^{t, x;\alpha} -\varepsilon\zeta\big) - \widehat{V}_{h,\varepsilon}\big(t+h^2, X_{t+h^2}^{t, x;\alpha}\big)  \right]}\quad \\
& \le  \Eb_{t}\left[ \widetilde{V}_{h,\varepsilon}\big(t+h^2 -\varepsilon^2\tau, X_{t+h^2}^{t, x;\alpha} -\varepsilon\zeta\big) - \widehat{V}_{h,\varepsilon}\big(t+h^2, X_{t+h^2}^{t, x;\alpha}\big)\right]
+ Ne^{NT} h^2\varepsilon.
\end{align*}
Substitution for the last term in \eqref{V-pre-moll} yields
 \begin{multline}
 \label{V-pre-moll2}
\widetilde{V}_{h,\varepsilon}\big(t-\varepsilon^2\tau, x-\varepsilon\zeta\big)\le\rho^g_{t, t+h^2}
\left[ \int_t^{t+h^2} \hspace{-0.5em} c(s, X_s^{t, x;u}, \alpha)\xdif s + \widehat{V}_{h,\varepsilon}\big(t+h^2, X_{t+h^2}^{t, x;\alpha}\big) \right] \\
{}+ \Eb_{t}\left[ \widetilde{V}_{h,\varepsilon}\big(t+h^2 -\varepsilon^2\tau, X_{t+h^2}^{t, x;\alpha} -\varepsilon\zeta\big) - \widehat{V}_{h,\varepsilon}\big(t+h^2, X_{t+h^2}^{t, x;\alpha}\big)\right]
+ Ne^{NT} h^2\varepsilon.
\end{multline}
We  now multiply
both sides of \eqref{V-pre-moll2} by $\varphi(\tau, \zeta)$ and integrate over $B$. By changing the order of integration in the expected value term of \eqref{V-pre-moll2} we observe that
\begin{multline*}
\int_B \Eb_{t}\left[  \widetilde{V}_{h,\varepsilon}\big(t+h^2 -\varepsilon^2\tau, X_{t+h^2}^{t, x;\alpha} -\varepsilon\zeta\big) - \widehat{V}_{h,\varepsilon}\big(t+h^2, X_{t+h^2}^{t, x;\alpha}\big) \right]
\varphi(\tau, \zeta)\xdif \tau\xdif \zeta \\
= \Eb_{t} \left[\int_B \left(  \widetilde{V}_{h,\varepsilon}\big(t+h^2 -\varepsilon^2\tau, X_{t+h^2}^{t, x;\alpha} -\varepsilon\zeta\big) - \widehat{V}_{h,\varepsilon}\big(t+h^2, X_{t+h^2}^{t, x;\alpha}\big)\right)
\varphi(\tau, \zeta)\xdif \tau\xdif \zeta  \right]= 0.
\end{multline*}
Other terms on the right hand side of \eqref{V-pre-moll2} do not depend on $(\tau,\zeta)$ and thus
\eqref{V-moll} follows.
\end{proof}

\section{Accuracy of the Approximation}
\label{s:accuracy}
We can now investigate the effect of the size of the discretization interval, $h^2$, on the accuracy of the value function approximation. For simplicity of presentation, we write $\sigma^\alpha(s,x)$ for $\sigma(s,x,\alpha)$ and $c^\alpha(s,x)$ for $c(s,x,\alpha)$

\begin{lmm}\label{lem:mp2} For any $w\in \Cc_b^{1,2}([0, T]\times\Rb^n)$, any
$ 0 \le t \le \theta \le T$, and all $u(\cdot)\in \mathcal{U}$ we have:
\begin{equation}
\label{estimate-Ito}
\begin{aligned}
w(t, x) = \rho^g_{t,\theta} \bigg[ \int_t^{\theta} c(s, X_s^{t, x; u}, u_s)\xdif s
+ w(\theta, X_{\theta}^{t,x; u}) - \bar{\zeta}\bigg],
\end{aligned}
\end{equation}
where
\begin{equation}
\label{zeta}
\bar{\zeta}  = \int_t^{\theta} \Big\{  \big[c^{u_s}+ \Lb^{u_s}w\big](s, X_s^{t,x;u})
+ g\big(s, [\mathcal{D}_x w\cdot\sigma^{u_s}](s, X_s^{t,x;u})\big)
\Big\}\xdif s.
\end{equation}
\end{lmm}

\begin{proof} For any $u(\cdot)\in \mathcal{U}$, we  apply It\^{o} formula
to $w(s, X_s^{t,x;u})$:
\begin{equation*}
w(\theta, X_{\theta}^{t,x;u}) -  w(t, x)  - \int_t^{\theta} \big[\Lb^{u_s}w\big](s, X_s^{t,x;u})\xdif s
= \int_t^{\theta} [\mathcal{D}_x w\cdot\sigma^{u_s}](s, X_s^{t,x;u})\xdif W_s.
\end{equation*}
Subtraction of $\int_t^{\theta} g\big(s, \big[\mathcal{D}_x w\cdot\sigma^{u_s}\big](s, X_s^{t,x;u})\big)\xdif s$ from both sides and evaluation of the risk on both sides yields
\begin{equation}
\begin{split}
\label{trick}
 \rho^g_{t,\theta} \bigg[ w(\theta, X_{\theta}^{s,x;u}) -  w(t,x)
- \int_t^{\theta} \Big( \big[\Lb^{u_s}w\big](s, X_s^{t,x;u}) + g\big(s, [\mathcal{D}_x w\cdot\sigma^{u_s}](s, X_s^{t,x;u})\big)\Big)\xdif s \bigg] \\
= \rho^g_{t,\theta} \bigg[ \int_t^{\theta} [\mathcal{D}_x w\cdot\sigma^{u_s}](s, X_s^{t,x,u})\xdif W_s -\int_t^{\theta} g\big(s, [\mathcal{D}_x w\cdot\sigma^{u_s}](s, X_s^{t,x;u})\big)\xdif s \bigg].
\end{split}
\end{equation}
The risk measure on the right hand side of \eqref{trick} is the solution of the following backward stochastic differential equation:
\begin{equation*}
\begin{split}
Y_t^{t,x;u} = \int_t^{\theta} [\mathcal{D}_x w\cdot\sigma^{u_s}](s, X_s^{t,x;u})\xdif W_s
 -\int_t^{\theta} g\big(s, [\mathcal{D}_x w\cdot\sigma^{u_s}](s, X_s^{t,x;u})\big)\xdif s \\
{} + \int_t^{\theta} g(s, Z_s^{t,x;u})\xdif s - \int_t^{\theta} Z_s^{t,x;u}\xdif W_s.
\end{split}
\end{equation*}
Substitution of $Z_s^{t,x;u} = [\mathcal{D}_x w\cdot\sigma^{u_s}](s, X_s^{t,x;u})$ yields $Y^{t,x;u}_t=0$. By the uniqueness of the solution of BSDE, the right hand side of \eqref{trick} is zero. Using the translation property on the left hand side of \eqref{trick}, we obtain
\begin{equation*}
w(t, x) = \rho^g_{t,\theta} \bigg[ - \int_t^{\theta} \Big(\big[\Lb^{u_s}w\big](s, X_s^{t,x;u}) + g\big(s, [\mathcal{D}_x w\cdot\sigma^{u_s}](s, X_s^{t,x;u})\big)\Big)\xdif s   + w\big(\theta, X_\theta^{t,x;u}\big)\bigg].
\end{equation*}
This is the same as \eqref{estimate-Ito}.
\end{proof}

The integral in \eqref{zeta} can be bounded by the following lemma.
\begin{lmm}
\label{l:estimate-alpha}
For all $t\in [0,T-h^2-\varepsilon^2]$, $x\in \Rb^n$, and all $\alpha\in U$, we have
\begin{equation}
\label{estimate-alpha}
[ c^\alpha
 + \Lb^{\alpha } \widehat{V}_{h,\varepsilon}](t,x)
+ g\big(t,[\mathcal{D}_x \widehat{V}_{h,\varepsilon} \cdot \sigma^\alpha](t,x)\big) \ge
-Ne^{NT}\left( \varepsilon + \frac{h}{\varepsilon^2}\right),
\end{equation}
where the constant $N$ does not depend on $h$, $\varepsilon$, and $T$.
\end{lmm}
\begin{proof}
By Lemma \ref{lem:mp1}, for every $\alpha\in U$ we have
\[
\widehat{V}_{h,\varepsilon}\big(t, x\big) \le \rho^g_{t, t+h^2}
\left[ \int_t^{t+h^2} \hspace{-0.5em}c(s, X_s^{t, x;u}, \alpha)\xdif s + \widehat{V}_{h,\varepsilon}\big(t+h^2 , X_{t+h^2}^{t, x;\alpha}\big) \right]+Ne^{NT} h^2\varepsilon.
\]
Using the translation property of $\rho_{t,t+h^2}$, we obtain the inequality:
\[
\rho^g_{{t},{t}+h^2}\left( \int_{t}^{t+h^2} c\big(s,X^{t, x;\alpha}_s,\alpha\big)\xdif s
+ \widehat{V}_{h,\varepsilon}\big(t+h^2,  X^{t, x;\alpha}_{t+h^2}\big)- \widehat{V}_{h,\varepsilon}(t, x)\right)\ge -Ne^{NT} h^2\varepsilon.
\]
Since $\widehat{V}_{h,\varepsilon}\in \Cc^{1,2}_{\textup{b}}([t,T-\varepsilon^2]\times\Rb^n)$, we can evaluate the  difference
$\widehat{V}_{h,\varepsilon}\big(t+h^2,  X^{t, x;\alpha}_{t+h^2}\big)-\widehat{V}_{h,\varepsilon}(t,x)$
by It\^{o} formula between $t$ and $t+h^2$:
\[
\widehat{V}_{h,\varepsilon}\big(t+h^2,  X^{t, x;\alpha}_{t+h^2}\big)
- \widehat{V}_{h,\varepsilon}(t, x) =  \int_{t}^{t+h^2}[\Lb^{\alpha} \widehat{V}_{h,\varepsilon}](s, X^{t, x;\alpha}_s) \xdif s
+ \int_{t}^{t+h^2} [\mathcal{D}_x \widehat{V}_{h,\varepsilon} \cdot \sigma^\alpha](s, X^{t, x;\alpha}_s) \xdif W_s.
\]
Substitution into the previous inequality yields:
\begin{equation}
\label{rho-in-HJBa}
\rho^g_{t,t+h^2}\Bigg( \int_{t}^{t+h^2} [ c^\alpha
 + \Lb^{\alpha } \widehat{V}_{h,\varepsilon}](s, X^{t, x;\alpha}_s) \xdif s
{} + \int_{t}^{t+h^2} [\mathcal{D}_x \widehat{V}_{h,\varepsilon} \cdot \sigma^\alpha](s, X^{t, x;\alpha}_s) \xdif W_s\Bigg)
\ge -Ne^{NT} h^2\varepsilon.
\end{equation}
The evaluation of the risk measure amounts to solving the following backward stochastic differential equation:
\begin{multline*}
\qquad Y_{t} = \int_{t}^{t+h^2} [ c^\alpha
 + \Lb^{\alpha } \widehat{V}_{h,\varepsilon}](s, X^{t, x;\alpha}_s) \xdif s
+ \int_{t}^{t+h^2} [\mathcal{D}_x \widehat{V}_{h,\varepsilon} \cdot \sigma^\alpha](s, X^{t, x;\alpha}_s) \xdif W_s \\
+ \int_{t}^{t+h^2} g(s,Z_s)\xdif s - \int_{t}^{t+h^2} Z_s\xdif W_s.\qquad
\end{multline*}
The equation has a unique solution:
\begin{gather*}
Z_s = [\mathcal{D}_x \widehat{V}_{h,\varepsilon} \cdot \sigma^\alpha](s, X^{t, x;\alpha}_s),\quad t \le s \le t+h^2, \\
Y_{t} = \int_{t}^{t+h^2} \Big\{[ c^\alpha
 + \Lb^{\alpha } \widehat{V}_{h,\varepsilon}](s, X^{t, x;\alpha}_s)
+ g\big(s,[\mathcal{D}_x \widehat{V}_{h,\varepsilon} \cdot \sigma^\alpha](s, X^{t, x;\alpha}_s)\big)\Big\} \xdif s.
\end{gather*}
We can thus write the inequality
\begin{align*}
Y_t \le {} & h^2 \left( [ c^\alpha
 + \Lb^{\alpha } \widehat{V}_{h,\varepsilon}](t,x)
+ g\big(t,[\mathcal{D}_x \widehat{V}_{h,\varepsilon} \cdot \sigma^\alpha](t,x)\big)\right)\\
&{} + h^2 \max_{t \le s \le t+h^2}\Eb_{t} \bigg\{ \Big|
[ c^\alpha
 + \Lb^{\alpha } \widehat{V}_{h,\varepsilon}](s, X^{t, x;\alpha}_s) -
[c^\alpha + \Lb^{\alpha } \widehat{V}_{h,\varepsilon}](t,x)\Big| \bigg\}\\
&{} + h^2 \max_{t \le s \le t+h^2}\Eb_{t} \bigg\{ \Big|
g\big(s,[\mathcal{D}_x \widehat{V}_{h,\varepsilon} \cdot \sigma^\alpha](s, X^{t, x;\alpha}_s)\big) -
g\big(t,[\mathcal{D}_x \widehat{V}_{h,\varepsilon} \cdot \sigma^\alpha](t,x)\big)\Big| \bigg\}.
\end{align*}
The last two terms can be bounded by $Ne^{NT}h^3 / \varepsilon^{2}$, owing to Assumption \ref{ass:sde} and Lemma \ref{estimatesall}. Combining this inequality with \eqref{rho-in-HJBa} and dividing by $h^2$, we conclude that for
all $\alpha\in U$ the estimate \eqref{estimate-alpha} is true.
\end{proof}

We are now ready to prove the main theorem of this section.

\begin{thrm}
\label{t:error-estimate}
Suppose Assumptions \Rref{a:riskmeasure} and \Rref{ass:sde} are satisfied. Then
for any $t\in[0, T]$, $x\in\Rb^n$, and $h\in(0, 1]$, we have
\begin{align*}
|V(t,x) - V_h(t, x)|\leq Ne^{NT}h^{\frac{1}{3}},
\end{align*}
where the constant $N$ depends only on $(K, n, d)$.
\end{thrm}
\begin{proof} We set $\varepsilon = h^{\frac{1}{3}}$ and organize the proof in three steps.

\emph{Step 1: } If $t\in [T-h^2-\varepsilon^2, T]$, then
 for any $u(\cdot)$ and some constant $C$ we have,
\begin{align*}
\big|V^{u}(t, x)-\varPhi(x)\big| &\leq
\rho_{t, T}\bigg(\int_t^T \big| c(s, X^{t, x, u}_s, u_s)\big|\xdif s + \big| \varPhi(X_{T}^{t, x, u}) - \varPhi(x)\big| \bigg)\\
&\leq K ( h^2+\varepsilon^2) + K \mathbb{E}_{t, x}\left[|X_{T}^{t, x, u} - x|\right]
\leq K(1+C)( h^2+\varepsilon^2) \le 2 K(1+C)h^{\frac{2}{3}}.
\end{align*}
In the above estimate we also used the fact that the solution of the forward--backward
system \eqref{s:cdp}--\eqref{s-fbsde} is Lipschitz in the initial condition \cite{MA}.
The same reasoning works for $V^{u}_h$, and thus
\begin{align*}
 |V^{u}_h(s, x) - \varPhi(x)|\leq 2 K(1+C)h^{\frac{2}{3}}.
\end{align*}
We can, therefore, for some constant $N$ write the inequality
\begin{align*}
|V^u(t, x) - V^u_h(t, x)|\leq Ne^{NT}h^{\frac{1}{3}}.
\end{align*}
The optimization over $u$ will not make it worse, and thus our assertion is true for these $t$.

\emph{Step 2: } Consider $t\in [0, T-h^2-\varepsilon^2]$.
By Lemma \ref{lem:mp2}, for all $u(\cdot)\in \mathcal{U}$ on $[t, T-h^2-\varepsilon^2]$, we have
\[
\widehat{V}_{h,\varepsilon}\big(t, x\big) \le
\rho^g_{t,T-h^2-\varepsilon^2}\left(\int_t^{T-h^2-\varepsilon^2}c(s, X_s^{t, x, u}, u_s)\xdif s + \widehat{V}_{h,\varepsilon}\big(T-h^2-\varepsilon^2, X_{T-h^2-\varepsilon^2}^{t, x, u}) - \bar{\zeta}\right),
\]
where, owing to Lemma \ref{l:estimate-alpha},
\[
\bar{\zeta} = \int_t^{T-h^2-\varepsilon^2} \left(\big[ c^{u_s} + \Lb^{u_s} \widehat{V}_{h,\varepsilon}\big](s, x) + g(s, [\partial_x \widehat{V}_{h,\varepsilon}\cdot\sigma^{u_s}](s, x))\right)\xdif s
 \ge -NTe^{NT}\left( \varepsilon + \frac{h}{\varepsilon^2}\right).
\]
These relations,  the monotonicity of the risk measure, and Lemmas  \ref{cor:be} and \ref{estimatesall} imply the estimate
\begin{multline*}
{V}_{h}(t, x) \le
\rho^g_{t,T-h^2-\varepsilon^2}\bigg(\int_t^{T-h^2-\varepsilon^2}c(s, X_s^{t, x, u}, u_s)\xdif s\\ {}  + {V}_{h}\big(T-h^2-\varepsilon^2, X_{T-h^2-\varepsilon^2}^{t, x, u}) \bigg)
{} +
NTe^{NT}\left( \varepsilon + \frac{h}{\varepsilon^2}\right)+4 Ne^{NT}\varepsilon.
\end{multline*}
In view of the inequality established in Step 1, using $\varepsilon = h^{\frac{1}{3}}$, and redefining $N$
appropriately, we obtain the following inequality:
\begin{equation}
\label{Vh-upper}
{V}_{h}(t, x) \le
\rho^g_{t,T-h^2-\varepsilon^2}\bigg(\int_t^{T-h^2-\varepsilon^2}c(s, X_s^{t, x, u}, u_s)\xdif s   + {V}\big(T-h^2-\varepsilon^2, X_{T-h^2-\varepsilon^2}^{t, x, u}) \bigg)
 +
Ne^{NT}h^{\frac{1}{3}}.
\end{equation}

\emph{Step 3: } We apply the dynamic programming equation \eqref{DPE} to the right hand side of
\eqref{Vh-upper} to conclude that
\begin{multline*}
{V}_{h}(t, x) \le
\inf_{u(\cdot)\in  \mathcal{U}}\rho^g_{t,T-h^2-\varepsilon^2}\bigg(\int_t^{T-h^2-\varepsilon^2}c(s, X_s^{t, x, u}, u_s)\xdif s \\{}  + {V}\big(T-h^2-\varepsilon^2, X_{T-h^2-\varepsilon^2}^{t, x, u}) \bigg)
{} +
Ne^{NT}h^{\frac{1}{3}}
 = {V}(t, x) +
Ne^{NT}h^{\frac{1}{3}},
\end{multline*}
as required.
\end{proof}


\begin{thebibliography}{10}

\bibitem{AT}
F.~Antonelli.
\newblock Backward-forward stochastic differential equations.
\newblock {\em Annals of Applied Probability}, 3:777 -- 793, 1993.

\bibitem{AP1}
P.~Artzner, F.~Delbaen, J.~M. Eber, and D.~Heath.
\newblock Thinking coherently.
\newblock {\em RISK}, 10:68--71, 1997.

\bibitem{AP2}
P.~Artzner, F.~Delbaen, J.~M. Eber, and D.~Heath.
\newblock Coherent measures of risk.
\newblock {\em Mathematical Finance}, 9:203--228, 1999.

\bibitem{BE1}
P.~Barrieu and N.~El~Karoui.
\newblock Optimal derivatve design under dynamic risk measures.
\newblock {\em Mathematics of Finance, Contemporary Mathematics}, 351:13--26,
  2004.

\bibitem{BE2}
P.~Barrieu and N.~El~Karoui.
\newblock Pricing, hedging and optimally designing derivatives via minimization
  of risk measures.
\newblock {\em Volume on Indifference Pricing, Princeton University Press},
  2009.

\bibitem{BN}
J.~Bion-Nadal.
\newblock Dynamic risk measures: time consistency and risk measures from bmo
  martingales.
\newblock {\em Finance and Stochastic}, 12:219--244, 2008.

\bibitem{BDM}
P.~Briand, B.~Delyon, and J.~M\'{e}min.
\newblock On the robustness of backward stochastic differential equations.
\newblock {\em Stochastic Processes and their Applications}, 97:229--253, 2002.

\bibitem{cavus2014risk}
\"{O}. \c{C}avu\c{s} and A.~Ruszczy\'nski.
\newblock Risk-averse control of undiscounted transient {M}arkov models.
\newblock {\em SIAM Journal on Control and Optimization}, 52(6):3935--3966,
  2014.

\bibitem{CD1}
P.~Cheridito, F.~Delbaen, and M.~Kupper.
\newblock Dynamic monetary risk measures for bounded discrete-time processes.
\newblock {\em Electronic Journal of Probability}, 11:57--106, 2006.

\bibitem{CK1}
P.~Cheridito and M.~Kupper.
\newblock Composition of time-consistent dynamic monetary risk measures in
  discrete time.
\newblock {\em International Journal of Theoretical and Applied Finance},
  14(1):137--162, 2011.

\bibitem{coquet2002filtration}
F.~Coquet, Y.~Hu, J.~M{\'e}min, and S.~Peng.
\newblock Filtration-consistent nonlinear expectations and related
  g-expectations.
\newblock {\em Probability Theory and Related Fields}, 123(1):1--27, 2002.

\bibitem{delbaen2010representation}
F.~Delbaen, S.~Peng, and E.~Rosazza~Gianin.
\newblock Representation of the penalty term of dynamic concave utilities.
\newblock {\em Finance and Stochastics}, 14(3):449--472, 2010.

\bibitem{DS1}
K.~Detlefsen and G.~Scandolo.
\newblock Conditional and dynamic convex risk measures.
\newblock {\em Finance and Stochastic}, 9:539--561, 2005.

\bibitem{fan2015dynamic}
J.~Fan and A.~Ruszczy{\'n}ski.
\newblock Dynamic risk measures for finite-state partially observable {M}arkov
  decision problems.
\newblock In {\em Proceedings of the Conference on Control and its
  Applications}, pages 153--158. SIAM, 2015.

\bibitem{FleSon}
W.~Flemming and H.~M. Soner.
\newblock {\em Controlled Markov Processes and Viscosity Solutions}.
\newblock Springer, 2006.

\bibitem{FP1}
H.~F\"{o}llmer and I.~Penner.
\newblock Convex risk measures and the dynamics of their penalty functions.
\newblock {\em Statistics and Decisions}, 24:61--96, 2006.

\bibitem{SF1}
H.~F\"{o}llmer and A.~Schied.
\newblock Convex measures of risk and trading constraints.
\newblock {\em Finance and Stochastic}, 6:429--447, 2002.

\bibitem{SF2}
H.~F\"{o}llmer and A.~Schied.
\newblock {\em Stochastic Finance: An Introduction in Discrete Time, 2nd ed}.
\newblock De Gruyter Berlin, 2004.

\bibitem{FS}
M.~Fritelli and G.~Scandolo.
\newblock Risk measures and capital requirements for processes.
\newblock {\em Math. Finance}, 16:589--612, 2006.

\bibitem{FR1}
M.~Frittelli and E.~Rosazza~Gianin.
\newblock Putting order in risk measures.
\newblock {\em Journal of Banking and Finance}, 26:1473--1486, 2002.

\bibitem{FR2}
M.~Frittelli and E.~Rosazza~Gianin.
\newblock Dynamic convex risk measures.
\newblock pages 227--248. Wiley, New York, 2004.

\bibitem{Krylov1}
N.~V. Krylov.
\newblock Approximating value funtions for controlled degenerate diffusion
  processes by using piece-wise constant polices.
\newblock {\em Electronic Journal Of Probability}, 4:1--19, 1999.

\bibitem{Krylov}
N.~V. Krylov.
\newblock {\em Controlled Diffusion Process (Stochastic Modelling and Applied
  Probability)}.
\newblock Springer, 2008.

\bibitem{laeven2014robust}
R.~J.~A. Laeven and M.~Stadje.
\newblock Robust portfolio choice and indifference valuation.
\newblock {\em Mathematics of Operations Research}, 39(4):1109--1141, 2014.

\bibitem{JLQW}
J.~Li and Q.~Wei.
\newblock Optimal control problems of fully coupled {FBSDE}s and viscosity
  solutions of {H}amilton--{J}acobi--{B}ellman equations.
\newblock {\em SIAM J. Control Optimization}, 52(3):1622--1662, 2014.

\bibitem{MA}
J.~Ma and J.~Yong.
\newblock {\em Forward-Backward Stochastic Differential Equations and their
  Applications}.
\newblock Springer, 2007.

\bibitem{oksendal2009maximum}
B.~{\O}ksendal and A.~Sulem.
\newblock Maximum principles for optimal control of forward-backward stochastic
  differential equations with jumps.
\newblock {\em SIAM Journal on Control and Optimization}, 48(5):2945--2976,
  2009.

\bibitem{oksendal2014forward}
B.~{\O}ksendal and A.~Sulem.
\newblock Forward--backward stochastic differential games and stochastic
  control under model uncertainty.
\newblock {\em Journal of Optimization Theory and Applications}, 161(1):22--55,
  2014.

\bibitem{PP1}
E.~Pardoux and S.~Peng.
\newblock Adapted solutions of backward stochastic differential equation.
\newblock {\em System and Control Letters}, 14:55--61, 1990.

\bibitem{PP2}
E.~Pardoux and S.~Peng.
\newblock Backward stochastic differential equations and quasilinear parabolic
  partial differential equations.
\newblock {\em Stochastic Differential Equations and Their Applications},
  176:200--217, 1992.

\bibitem{PT}
E.~Pardoux and S.~Tang.
\newblock Forward-backward stochastic differential equations and quasilinear
  prarabolic pdes.
\newblock {\em Probability Theory and Related Fields}, 114:123 -- 150, 1999.

\bibitem{PSG2}
S.~Peng.
\newblock Backward {SDE} and related $g$-expectation.
\newblock {\em Backward Stochastic Differential Equations, Pitman Research
  Notes Mathematics}, 364:141--159, 1997.

\bibitem{PSG1}
S.~Peng.
\newblock Nonlinear expectations, nonlinear evaluations and risk measures.
\newblock {\em Lecture Notes in Mathematics, Springer}, 2004.

\bibitem{peng1999fully}
S.~Peng and Z.~Wu.
\newblock Fully coupled forward-backward stochastic differential equations and
  applications to optimal control.
\newblock {\em SIAM Journal on Control and Optimization}, 37(3):825--843, 1999.

\bibitem{Pham}
H.~Pham.
\newblock {\em Continuous-time Stochastic Control and Optimization with
  Financial Application(Stochastic Modelling and Applied Probability)}.
\newblock Springer, 2010.

\bibitem{quenez2013bsdes}
M.-C. Quenez and A.~Sulem.
\newblock Bsdes with jumps, optimization and applications to dynamic risk
  measures.
\newblock {\em Stochastic Processes and their Applications}, 123(8):3328--3357,
  2013.

\bibitem{RF}
F.~Riedel.
\newblock Dynamic coherent risk measures.
\newblock {\em Stochastic Processes and their Applications}, 112:185--200,
  2004.

\bibitem{gianin2006risk}
E.~Rosazza~Gianin.
\newblock Risk measures via g-expectations.
\newblock {\em Insurance: Mathematics and Economics}, 39(1):19--34, 2006.

\bibitem{AR4}
A.~Ruszczy\'{n}ski.
\newblock Risk-averse dynamic programming for {M}arkov decision process.
\newblock {\em Mathematical Programming Series B}, 125:235--261, 2010.

\bibitem{AR2}
A.~Ruszczy\'{n}ski and A.~Shapiro.
\newblock Conditional risk mappings.
\newblock {\em Mathematics of Operations Research}, 31:544--561, 2006.

\bibitem{AR1}
A.~Ruszczy\'{n}ski and A.~Shapiro.
\newblock Optimization of convex risk functions.
\newblock {\em Mathematics of Operations Research}, 31:433--452, 2006.

\bibitem{AR3}
A.~Ruszczy\'{n}ski, A.~Shapiro, and D.~Dentcheva.
\newblock {\em Lectures on Stochastic Programming. Modeling and Theory}.
\newblock SIAM-Society for Industrial and Applied Mathematics, 2009.

\bibitem{ARYAO}
A.~Ruszczy\'nski and J.~Yao.
\newblock A risk-averse analog of the {H}amilton-{J}acobi-{B}ellman equation.
\newblock In {\em Proceedings of the SIAM Conference on Control and its
  Applications, Paris}, pages 462--468. 2015.

\bibitem{ST}
M.~Stadje.
\newblock Extending dynamic convex risk measures from discrete time to
  continuous time: a convergence approach.
\newblock {\em Insurance: Mathematics and Economics}, 47:391--404, 2010.

\bibitem{Yong}
J.~Yong.
\newblock Finding adapted solutions of forward-backward stochastic differential
  equations: method of continuation.
\newblock {\em Probability Theory and Related Fiedls, Springer}, 1997.

\bibitem{Zhang}
J.~Zhang.
\newblock The well posedness of {FBSDE}.
\newblock {\em Discrete and Continuous Dynamical Systerms Series B}, 6:927 --
  940, 2006.

\bibitem{XYZ}
X.~Zhou and J.~Yong.
\newblock {\em Stochastic Control - Hamiltonian Systems and HJB Equations}.
\newblock Springer, 1998.

\end{thebibliography}
\end{document}